\documentclass[reqno,a4paper,11pt]{amsart}
\usepackage[all,poly]{xy}
\usepackage{amsfonts}
\usepackage[mathcal]{eucal}
\usepackage{amssymb}
\usepackage{amsmath}
\usepackage{mathrsfs}
\usepackage{amsthm}
\usepackage{color}
\usepackage[pagebackref,colorlinks]{hyperref}
\usepackage{enumerate}

\theoremstyle{plain}
\newtheorem{lemma}{Lemma}[section] 
\newtheorem{theorem}[lemma]{Theorem}
\newtheorem{corollary}[lemma]{Corollary}
\newtheorem{proposition}[lemma]{Proposition}

\theoremstyle{definition}
\newtheorem{remark}[lemma]{Remark}
\newtheorem{example}[lemma]{Example}
\newtheorem{definition}[lemma]{Definition}

\newcommand{\M}{\operatorname{\mathbb M}}
\newcommand{\ol}{\overline}

\newcommand{\Zset}{\mathbb Z}

\newcommand{\sto}{\to^{\hskip-.42cm A1}}
\newcommand{\cto}{\rightsquigarrow}
\newcommand{\supp}{\operatorname{supp}}

\newcommand{\so}{\mathbf{s}}
\newcommand{\ra}{\mathbf{r}}

\newcommand{\V}{\mathcal V}

\title[Comparability in the graph monoid]{Comparability in the graph monoid}

\author{Roozbeh Hazrat}
\address{Centre for Research in Mathematics and Data Science, Western Sydney University, Australia} 
\email{r.hazrat@westernsydney.edu.au}

\author{Lia Va\v s}
\address{Department of Mathematics, Physics and Statistics, University of the Sciences, Philadelphia, PA 19104, USA}
\email{l.vas@usciences.edu}


\thanks{The authors are very grateful to the referee for a prompt, detailed and thoughtful report. The first author would also like to acknowledge Australian Research Council grant DP160101481.}

\begin{document}
 
\begin{abstract} 
Let $\Gamma$ be the infinite cyclic group on a generator $x.$ To avoid confusion when working with $\mathbb Z$-modules which also have an additional $\mathbb Z$-action, we consider the $\mathbb Z$-action to be a $\Gamma$-action instead.  

Starting from a directed graph $E$, one can define a cancellative commutative monoid $M_E^\Gamma$ with a $\Gamma$-action which agrees with the monoid structure and a natural order. The order and the action enable one to label each nonzero element as being exactly one of the following: comparable (periodic or aperiodic) or incomparable. We comprehensively pair up these element features with the graph-theoretic properties of the generators of the element. We also characterize graphs such that every element of $M_E^\Gamma$ is comparable, periodic, graphs such that every nonzero element of $M_E^\Gamma$ is aperiodic, incomparable, graphs such that no nonzero element of $M_E^\Gamma$ is periodic, and graphs such that no element of $M_E^\Gamma$ is aperiodic.

The Graded Classification Conjecture can be formulated to state that $M_E^\Gamma$ is a complete invariant of the Leavitt path algebra $L_K(E)$ of $E$ over a field $K.$ Our characterizations indicate that the Graded Classification Conjecture may have a positive answer since the properties of $E$ are well reflected by the structure of $M_E^\Gamma.$  Our work also implies that some results of \cite{Talented_monoid} hold without requiring the graph to be row-finite.  
\end{abstract}
 
\maketitle

\section{Introduction}

There are several different ways to associate an algebra over a field $K$ to a directed graph $E$. For example, one can form the path algebra $P_K(E)$ which is a vector space over $K$ based on paths multiplied using concatenation. If one wants to add a natural involutive structure to this algebra (as, for example, when completing the path algebra over complex numbers to obtain the graph $C^*$-algebra $C^*(E)$), then every vertex naturally becomes a self-adjoint idempotent, a projection, and every edge $e$ becomes a partial isometry making the projections $ee^*$ and $e^*e$ equivalent. If $\so$ and $\ra$ are the source and range maps of $E$ respectively, and $\so(e)=v,$ then $ve=e$ so that $vee^*=ee^*$ and, hence, $v\geq ee^*$
(recall that the projections are ordered by $p\leq q$ if $pq=p$). On the other hand, if $w=\ra(e),$ then $ew=e$ and so $w\geq e^*e.$ The requirement that $w=e^*e$ is called the (CK1) axiom. One also aims to have that the projections $v$ and $w$ are equivalent if $e$ is the {\em only} edge from $v$ to $w$ and if $v$ does not emit any other edges. This is achieved by an additional requirement, the (CK2) axiom, stating that $v=\sum_{e\in \so^{-1}(v)} ee^*$ if $v$ emits at least one and only finitely many edges. The axioms (CK1) and (CK2) imposed on the involutive closure of the path algebra produce the Leavitt path algebra $L_K(E).$ If $\V(L_K(E))$ is the monoid of the isomorphism classes of finitely generated projective modules (or conjugation classes of idempotent matrices), the (CK1) and (CK2) axioms imply that  
\[[v]=\sum_{e\in \so^{-1}(v)} [\ra(e)]\]
holds in  $\V(L_K(E))$ for every vertex $v$ which emits at least one and only finitely many edges. 
If $E$ is such that every vertex emits only finitely many edges, in which case we say that $E$ is row-finite, one of the first papers on Leavitt path algebras \cite{Ara_Moreno_Pardo} shows that elements $[v]$ generate $\V(L_K(E))$ and that the above relations are the {\em only} relations which hold on $\V(L_K(E))$. Thus, to capture $\V(L_K(E))$ entirely, it is sufficient to consider a free commutative monoid $M_E,$ called the {\em graph monoid}, generated by $[v]$ where $v$ is a vertex of $E$ subject to the above relations. In \cite{Ara_Goodearl}, the authors generalized this construction to arbitrary graphs. 
To handle vertices which emit infinitely many edges (infinite emitters), one adds two natural relations to the one listed above (the details are reviewed in Section \ref{subsection_graph_group_nongraded_case}) to obtain $M_E.$ 

The monoid $M_E$ is not necessarily cancellative which is easy to see: if $v$ is a vertex emitting two edges to itself, then the relation $[v]+[v]=[v]$ holds in the monoid but the generator $[v]$ is nonzero. So, when one forms the Grothendieck group $G_E$ of the monoid $M_E$ a lot of information can get lost. In particular, if $E$ is a graph consisting only of the vertex and edges from the previous example, then $G_E=0.$ 

In addition to the above mentioned downside, very different graphs give rise to isomorphic monoids and, consequently, isomorphic Grothendieck groups. For example, $\bullet$ and 
$\xymatrix{ {\bullet} \ar@(ur,dr)}\;\;\;\;\;.$ In addition, consider the graphs $E_1$ and $E_2$ below, for example.  
\[
\xymatrix{\bullet^{v_1} \ar[r]                  &\bullet^{w_1}}\hskip2cm
\xymatrix{\bullet^{v_2} \ar[r] & \bullet \ar[r] &\bullet^{w_2}}\]
The relation $[v_1]=[w_1]$ holds in the first and the relation $[v_2]=[w_2]$ holds in the second graph monoid regardless of the fact that the length of the only path from $v_1$ to $w_1$ is 1 in $E_1$ while the length of the only path from $v_2$ to $w_2$ is 2 in $E_2.$ So, this type of information is also lost in the Grothendieck group. 

These downsides can be avoided by taking the natural grading of a Leavitt path algebra into consideration. Namely, the elements $pq^*$ where $p$ and $q$ are paths, generate the entire algebra as a $K$-vector space and if $p$ and $q$ are such that the difference of the length of $p$ and the length of $q$ is an integer $n,$ the generator $pq^*$ is considered to be in the $n$-th component of $L_K(E).$ This produces a $\Zset$-graded structure of $L_K(E)$ where $\Zset$ is the set of integers. For a ring $R$ graded by a group $\Gamma,$ the monoid $\V^\Gamma(R)$ of the graded isomorphism classes of finitely generated {\em graded} projective modules (or conjugation classes of certain homogeneous idempotent matrices) is a natural analogue of $\V(R).$ The monoid $\V^\Gamma(R)$ has a canonical $\Gamma$-action and we refer to a monoid with this type of structure as a $\Gamma$-monoid. 

To avoid confusion when working with structures which are $\Zset$-modules but also have an additional $\Zset$-action, we let $\Gamma=\{x^n\mid n\in \Zset\}$ and consider the $\Zset$-action to be a $\Gamma$-action instead. The $\Gamma$-action on $\V^\Gamma(L_K(E))$ is such that the relation $[v]=\sum_{e\in \so^{-1}(v)} [\ra(e)]$ becomes 
\[[v]=\sum_{e\in \so^{-1}(v)} x[\ra(e)]\] if $\so^{-1}(v)$ is nonempty and finite. The power 1 of $x$ in this relation indicates the length of the path $e$ from $v$ to $\ra(e).$ With analogous modifications of the other defining relations, we let {\em the graph $\Gamma$-monoid $M_E^\Gamma$} be the quotient of a free $\Gamma$-monoid $F_E^\Gamma$ with basis elements labeled by the vertices and the elements related to the infinite emitters subject to the defining relations (Section \ref{subsection_graph_group} contains more details).  
Alternatively, if $\to_1$ is a binary relation of $F_E^\Gamma$ given by these defining relations, $\to$ is the  reflexive and transitive closure of $\to_1,$ and $\sim$ is the congruence closure of $\to,$ then $M_E^\Gamma$ is the quotient $\Gamma$-monoid $F_E^\Gamma/\sim.$  The $\Gamma$-monoid $M_E^\Gamma$ is naturally isomorphic to $\V^\Gamma(L_K(E))$.

The monoid $M_E^\Gamma$ has several important advantages over $M_E.$ First, it is always cancellative by \cite[Corollary 5.8]{Ara_et_al_Steinberg} (we give an alternative proof in Proposition \ref{cancellative}) and so it is exactly the positive cone of its Grothendieck group $G_E^\Gamma.$ This group inherits the $\Gamma$-action from $M_E^\Gamma$ so we refer to it as the Grothendieck $\Gamma$-group. Second, the information on the lengths of paths from a vertex to vertex is not lost. For example, if $E_1$ and $E_2$ are the above two graphs, the relations $[v_1]=[w_1]$ and $[v_2]=[w_2]$ of $M_{E_1}$ and $M_{E_2}$ become 
\begin{center}
$[v_1]=x[w_1]\;\;$ and $\;\;[v_2]=x^2[w_2]$  
\end{center}
in $M_{E_1}^\Gamma$ and $M_{E_2}^\Gamma$ respectively. 
Here, the powers of $x$ indicate that the length of the (only) path from $v_1$ to $w_1$ is 1 in $E_1$ and that the length of the (only) path from $v_2$ to $w_2$ is 2 in $E_2.$ In addition, very different graphs $\bullet$ and  $\xymatrix{ {\bullet} \ar@(ur,dr)}\;\;\;\;\;$ have different Grothendieck $\Gamma$-groups: $G_E^\Gamma$ of the first graph is isomorphic to $\Zset[\Gamma]$ with the natural action of $\Gamma$ while $G_E^\Gamma$ of the second graph is isomorphic to $\Zset$ with the trivial action of $\Gamma.$ 

Because of these favorable properties of $M_E^\Gamma$ and $G_E^\Gamma,$ it was conjectured in \cite{Roozbeh_Annalen} that $G_E^\Gamma,$ considered with a natural pre-order and an order-unit, is a complete invariant of a row-finite graph $E.$ Since the monoid $M_E^\Gamma$ is always cancellative, this conjecture can also be phrased in terms of $M_E^\Gamma$ instead of $G_E^\Gamma.$ In addition, the restriction on row-finiteness can be deleted and we refer to the following statement as the {\em Graded Classification Conjecture.}   
\begin{itemize}
\item[] For any two graphs $E$ and $F$ and any field $K$, $L_K(E)$ and $L_K(F)$ are isomorphic as $\Gamma$-graded algebras if and only if $M_E^\Gamma$ and $M_F^\Gamma$ are isomorphic as pre-ordered $\Gamma$-monoids with order-units.  
\end{itemize}

Since $M_E^\Gamma$ is cancellative, the natural pre-order is, in fact, an order. In \cite{Talented_monoid}, the authors show that the relation $a< x^na$ is impossible for any $a\in M_E^\Gamma$ and any positive integer $n$ if $E$ is row-finite. In Proposition \ref{generalization_of_lemma_4_1}, we show that this holds for all graphs $E.$ Hence, there are two remaining cases. 
\begin{enumerate}
\item $a\geq x^na$ for some positive integer $n.$ In this case, we say that $a$ is {\em comparable}.

\item  $a$ and $x^na$ incomparable for any positive integer $n.$ In this case, we say that $a$ is {\em incomparable}.
\end{enumerate}
If $a$ is comparable, there are two possibilities.
\begin{enumerate}
\item[(1i)] $a=x^na$ for some positive integer $n.$ In this case, we say that $a$ is {\em periodic}. 
 
\item[(1ii)] $a>x^na$ for some positive integer $n.$ In this case, we say that $a$ is {\em aperiodic}. 
\end{enumerate}
In this paper, we provide complete characterizations of all four types of elements (comparable, incomparable, periodic and aperiodic) in terms of the graph-theoretic properties of the generators of an element. We obtain this by three groups of results. First, in Section \ref{section_connectivity}, we obtain a graph-theoretic characterization of the relation $\to$ (Proposition \ref{connecting}). Second, in Sections \ref{subsection_stationary} and \ref{subsection_core_lemma}, we introduce and study certain well-behaved building blocks of comparable elements, the stationary elements. Third, in Section \ref{subsection_partition}, we produce a graph-theoretic characterization of a stationary element (Proposition \ref{key_lemma}). This enables us to prove Theorem \ref{comparable}, the main result of Section \ref{section_comparable}, which characterizes a comparable element in terms of the graph-theoretic properties of its generators.  

In Section \ref{section_other_three}, we characterize periodic and aperiodic elements in Theorems \ref{periodic} and \ref{aperiodic}. We have already found a use of Theorem \ref{periodic}: it was used in \cite[Theorem 3.1]{Crossed_product} to characterize Leavitt path algebras which are crossed products in terms of the properties of the underlying graphs. We also characterize graphs such that  {\em every} element of $M_E^\Gamma$ is comparable (Theorem  \ref{all_comparable}), periodic (Theorem \ref{all_periodic}), graphs such that every nonzero element of $M_E^\Gamma$ is aperiodic (Theorem \ref{all_aperiodic}), incomparable (Corollary \ref{all_incomparable}), graphs such that no nonzero element of $M_E^\Gamma$ is periodic (Corollary \ref{no_periodic}), and graphs such that no element of $M_E^\Gamma$ is aperiodic (Corollary \ref{no_aperiodic}). These characterizations comprehensively pair up the monoid and the graph properties and are summarized in the table below. In the table, c$(a)$, p$(a)$, ap$(a)$, ic$(a)$ shorten the statements that $a\in M_E^\Gamma$ is comparable, periodic, aperiodic, and incomparable respectively. The formula ``$(\exists a\neq 0)$ c$(a)$'', for example, shortens ``There is a nonzero comparable element in $M_E^\Gamma$''. 

\begin{center}
\begin{tabular}{|l|l|}\hline
{\bf Property of the graph $\Gamma$-monoid} & {\bf Property of the graph}\\ \hline\hline
\hskip1cm $(\exists a\neq 0)$ c$(a)\;\;$ = $\;(\exists a\neq 0)$ not ic$(a)$  & There is a cycle.\\ \hline
\hskip1cm $(\forall a\neq 0)$ ic$(a)\;$ = $\;(\forall a\neq 0)$ not c$(a)$  & There is no cycle. \\ \hline
\hskip1cm $(\exists a\neq 0)$ p$(a)$   & There is a cycle with no exits. \\ \hline
\hskip1cm $(\exists a)\;\;\;\;\;\;\,$ ap$(a)$  & There is a cycle with an exit.\\ \hline
\hskip1cm $(\forall a)\;\;\;\;\;\;\,$ c$(a)\;\;$ = $\;(\forall a)\;\;$ not ic$(a)$  & Condition from Theorem \ref{all_comparable} holds.\\ \hline
\hskip1cm $(\forall a)\;\;\;\;\;\;\,$ p$(a)$   & Condition from Theorem \ref{all_periodic} holds.\\ \hline
\hskip1cm $(\forall a\neq 0)$ ap$(a)$  & Condition from Theorem \ref{all_comparable} holds\\ & and every cycle has an exit.\\ \hline
\hskip1cm $(\forall a\neq 0)$ not p$(a)$   & Every cycle has exits.\\ \hline
\hskip1cm $(\forall a)\;\;\;\;\;\;\;$ not ap$(a)$  & No cycle has exits.\\ \hline
\hskip1cm $(\exists a)\;\;\;\;\;\;\;$ ic$(a)\;\;$ = $(\exists a)\;\,$ not c$(a)$  &  Condition from Theorem \ref{all_comparable} fails.\\ \hline
\hskip1cm $(\exists a)\;\;\;\;\;\;\;$ not p$(a)$   & Condition from Theorem \ref{all_periodic} fails.\\ \hline
\hskip1cm $(\exists a\neq 0)\,$ not ap$(a)$  & Condition from Theorem \ref{all_comparable} fails\\ & or there is a cycle with no exits.\\ \hline
\end{tabular}
\end{center}

In Section \ref{subsection_talented}, we relax the assumptions of statements in \cite{Talented_monoid}. In particular, we show that the main results of \cite{Talented_monoid} hold without the requirement that the graph is row-finite (Corollaries \ref{talented_corollary1}, \ref{talented_corollary2}, \ref{talented_corollary3} and the first part of Corollary \ref{talented4}). The second part of Corollary \ref{talented4} lists further properties of graphs which are preserved if the graph $\Gamma$-monoids are isomorphic. 

Our work focuses on graphs and their graph $\Gamma$-monoids. Leavitt path algebras, often mentioned in the introduction to illustrate wider context, do not appear often in the rest of the paper and no prior knowledge of Leavitt path algebras is needed for understanding our main results.

\section{Prerequisites, notation and preliminaries}\label{section_prerequisites}

In this section only, we use $\Gamma$ to denote an arbitrary group with multiplicative notation. In the other sections of the paper, $\Gamma$ stands for the infinite cyclic group generated by an element $x.$ 

\subsection{Pre-ordered \texorpdfstring{$\Gamma$}{TEXT}-monoids and \texorpdfstring{$\Gamma$}{TEXT}-groups}\label{subsection_preordered_groups}
If $M$ is an additive monoid with a left action of $\Gamma$ which agrees with the monoid operation, we say that $M$ is a {\em $\Gamma$-monoid}. If $G$ an abelian group with a left action of $\Gamma$ which agrees with the group operation, we say that $G$ is a {\em $\Gamma$-group}. Such action of $\Gamma$ uniquely determines a left $\Zset[\Gamma]$-module structure on $G,$ so $G$ is also a left $\Zset[\Gamma]$-module.  

Let $\geq$ be a reflexive and transitive relation (a pre-order) on a $\Gamma$-monoid $M$ ($\Gamma$-group $G$) such that $g_1\geq g_2$ implies $g_1 + h\geq g_2 + h$ and $\gamma g_1 \geq \gamma g_2$ for all $g_1, g_2, h$ in $M$ (in $G$) and $\gamma\in \Gamma.$ We say that such monoid $M$ is a {\em pre-ordered $\Gamma$-monoid} and that such a group $G$ is a {\em pre-ordered $\Gamma$-group}. 

If $G$ is a pre-ordered $\Gamma$-group, the set $G^+=\{x\in G\mid x\geq 0\},$ called the positive cone of $G,$ is a $\Gamma$-monoid. Any additively closed subset $M$ of $G$ which contains 0 and is closed under the action of $\Gamma,$ defines a pre-order $\Gamma$-group structure on $G$ such that $G^+=M$. Such set $G^+$ is {\em strict} if $G^+\cap (-G^+)=\{0\}$ and this condition is equivalent with the pre-order being a partial order. In this case, we say that $G$ is an {\em ordered $\Gamma$-group}. For example, $\Zset[\Gamma]$ is an ordered $\Gamma$-group with the positive cone $\Zset^+[\Gamma]$ consisting of elements $a=\sum_{i=1}^n k_i\gamma_i\in \Zset[\Gamma]$ such that $k_i\geq 0$ for all $i=1,\ldots, n.$ 

An element $u$ of a pre-ordered $\Gamma$-monoid $M$ is an \emph{order-unit} if for any $x\in M$, there is a nonzero $a\in \Zset^+[\Gamma]$ such that $x\leq au.$ An element $u$ of a pre-ordered $\Gamma$-group $G$ is an \emph{order-unit} if $u\in G^+$ and for any $x\in G$, there is a nonzero $a\in \Zset^+[\Gamma]$ such that $x\leq au.$ 

If $G$ and $H$ are pre-ordered $\Gamma$-groups, a $\Zset[\Gamma]$-module homomorphism $f\colon G\to H$ is {\em order-preserving} or {\em positive} if $f(G^+)\subseteq H^+.$ If $G$ and $H$ are pre-ordered $\Gamma$-groups with order-units $u$ and $v$ respectively, an order-preserving $\Zset[\Gamma]$-module homomorphism $f\colon G\to H$ is {\em order-unit-preserving} if $f(u)=v.$ 

A {\em $\Gamma$-order-ideal} of a pre-ordered $\Gamma$-monoid $M$ is a $\Gamma$-submonoid $I$ of $M$ such that $a\leq b$ and $b\in I$ implies $a\in I.$ If $G$ is a pre-ordered $\Gamma$-group, a $\Gamma$-subgroup $J$ of $G$ is a {\em $\Gamma$-order-ideal} of $G$ if $J\cap G^+$ is a $\Gamma$-order-ideal of $G^+$ and $J=\{x-y\mid x,y\in J\cap G^+\}$ (equivalently, $J$ is a directed and convex $\Gamma$-subgroup of $G$ using definitions of a directed set and a convex set from \cite{Goodearl_interpolation_groups_book}). The lattices of $\Gamma$-order-ideals of $G^+$ and $\Gamma$-order-ideals of $G$ are isomorphic by the map $I\mapsto \{x-y \mid x,y\in I\}$ with the inverse $J\mapsto J\cap G^+.$

\subsection{Graded rings}\label{subsection_graded_rings}
We briefly review the concept of graded rings for context only. Other than a part of the statement of Corollary \ref{talented_corollary3}, 
no result of this paper refers to graded rings or requires any knowledge of their properties.   

A ring $R$ is \emph{$\Gamma$-graded} if $R=\bigoplus_{ \gamma \in \Gamma} R_{\gamma}$ where $R_{\gamma}$ is an additive subgroup of $R$ and $R_{\gamma}  R_{\delta} \subseteq R_{\gamma\delta}$ for all $\gamma, \delta \in \Gamma$. The standard definitions of graded right $R$-modules, graded module homomorphisms and isomorphisms, and graded projective right modules can be found in \cite{NvO_book} and \cite{Roozbeh_book}. If $M$ is a graded right $R$-module and $\gamma\in\Gamma,$ the $\gamma$-\emph{shifted} graded right $R$-module $(\gamma)M$ is defined as the module $M$ with the $\Gamma$-grading given by $(\gamma)M_\delta = M_{\gamma\delta}$ for all $\delta\in \Gamma.$ 

If $R$ is a $\Gamma$-graded ring, let $\V^{\Gamma}(R)$ denote the monoid of graded isomorphism classes $[P]$ of finitely generated graded projective right $R$-modules $P$ with the direct sum as the addition operation and the left $\Gamma$-action given by $(\gamma, [P])\mapsto [(\gamma^{-1})P].$\footnote{If $M$ is a graded left $R$-module and $\gamma\in\Gamma,$ the $\gamma$-shifted graded left $R$-module $M(\gamma)$ is the module $M$ with the $\Gamma$-grading given by $M(\gamma)_\delta = M_{\delta\gamma}$ for all $\delta\in \Gamma.$  The monoid $\V^\Gamma(R)$ can be represented using the classes of left modules in which case the corresponding formula is $(\gamma, [P])\mapsto [P(\gamma)].$ Two representations are equivalent (see \cite[Section 2.4]{NvO_book} or  \cite[Section 1.2.3]{Roozbeh_book}). } In particular, the definitions and results of \cite[\S 3.2]{Roozbeh_book} carry to the case when $\Gamma$ is not necessarily abelian as it is explained in \cite[Section 1.3]{Lia_realization}.
The \emph{Grothendieck $\Gamma$-group}  $K_0^{\Gamma}(R)$ is defined as the group completion of  the $\Gamma$-monoid $\V^{\Gamma}(R)$ which naturally inherits the action of $\Gamma$ from $\V^{\Gamma}(R)$.  The monoid $\V^{\Gamma}(R)$ is a pre-ordered $\Gamma$-monoid and the group  $K_0^{\Gamma}(R)$ is a pre-ordered $\Gamma$-group for any $\Gamma$-graded ring $R$. If $\Gamma$ is the trivial group, $K_0^{\Gamma}(R)$ is the usual $K_0$-group.

\subsection{Graphs}\label{subsection_graphs}
If $E$ is a directed graph, let $E^0$ denote the set of vertices, $E^1$ the set of edges and $\so$ and $\ra$ the source and the range maps of $E.$ The graph $E$ is {\em finite} if both $E^0$ and $E^1$ are finite and $E$ is {\em row-finite} if $\so^{-1}(v)$ is finite for every $v\in E^0.$ A vertex $v\in E^0$ is a {\em sink} if $\so^{-1}(v)=\emptyset$ and a {\em source} if $\ra^{-1}(v)=\emptyset.$ A vertex of $E$ is {\em regular} if  $\so^{-1}(v)$ is finite and nonempty. 

We use the standard definitions of a {\em path}, a {\em closed simple path} and a {\em cycle} (see \cite[Definitions 1.2.2. and 2.0.2]{LPA_book}).
A path $q$ is a {\em prefix} of a path $p$ if $p=qr$ for some path $r.$ If $q=\so(p),$ then $q$ is a {\em trivial} prefix. If $r\neq \ra(p),$ then $q$ is a {\em proper} prefix. If $E$ has no cycles, $E$ is {\em acyclic}. A cycle $c$ {\em has an exit} if a vertex on $c$ emits an edge outside of $c.$ The graph $E$ satisfies {\em Condition (NE)} (and $E$ is a {\em no-exit graph} in this case) if $v$ emits just one edge for every vertex $v$ of every cycle. The graph $E$ satisfies {\em Condition (L)} if every cycle has an exit (equivalently if every closed simple path has an exit) and $E$ satisfies {\em Condition (K)} if for each vertex $v$ which lies on a closed simple path, there are at least two different closed simple paths based at $v$. An {\em infinite path} is a sequence of edges $e_1e_2\ldots$ such that $\ra(e_i)=\so(e_{i+1})$ for $i=1,2,\ldots$. Such infinite path {\em ends in a cycle} if there is a positive integer $n$ and a cycle $c$ such that $e_ne_{n+1}\hdots$ is equal to $cc\hdots.$  

If $E$ is a finite and acyclic graph, it is well-established that it has a source. Since we were not aware of a reference for this fact and we use it in the proof of Lemma \ref{core_lemma}, we provide a quick proof for it. 

\begin{lemma}
If $E$ is a finite and acyclic graph, it has a source. 
\label{sources}
\end{lemma}
\begin{proof}
If the graph $E$ does not have any edges, then each of its vertices is both a source and a sink. If $E$ has edges, pick any of them, say $e_0.$ If $\ra^{-1}(\so(e_0))$ is empty, then $\so(e_0)$ is a source. If $\ra^{-1}(\so(e_0))$ is nonempty, take $e_1\in \ra^{-1}(\so(e_0)).$ Then $e_0\neq e_1$ since otherwise $\ra(e_0)=\so(e_0)$ and $e_0$ would be a cycle. If $\ra^{-1}(\so(e_1))$ is empty, then $\so(e_1)$ is a source. If $\ra^{-1}(\so(e_1))$ is nonempty, continue the process. At any step of the process, we obtain a different edge than any of the edges considered previously otherwise $E$ has a cycle. Since $E$ is finite, this process eventually ends. If it ends at the $n$-th step, then $\so(e_n)$ is a source. 
\end{proof}

\subsection{Leavitt path algebras}\label{subsection_LPAs}
We review the concept of a Leavitt path algebra for context only. No result of this paper except one part of Theorem \ref{all_periodic} refers to  Leavitt path algebras or requires any knowledge of these algebras.   
If $K$ is any field, the \emph{Leavitt path algebra} $L_K(E)$ of $E$ over $K$ is a free $K$-algebra generated by the set  $E^0\cup E^1\cup\{e^* \mid  e\in E^1\}$ such that, for all vertices $v,w$ and edges $e,f,$

\begin{tabular}{ll}
(V)  $vw =0$ if $v\neq w$ and $vv=v,$ & (E1)  $\so(e)e=e\ra(e)=e,$\\
(E2) $\ra(e)e^*=e^*\so(e)=e^*,$ & (CK1) $e^*f=0$ if $e\neq f$ and $e^*e=\ra(e),$\\
(CK2) $v=\sum_{e\in \so^{-1}(v)} ee^*$ for each regular vertex $v.$ &\\
\end{tabular}

By the first four axioms, $L_K(E)$ is a $K$-linear span of the elements of the form $pq^*$ for paths $p$ and $q.$ 
If $L_K(E)_n$ is the $K$-linear span of $pq^*$ for paths $p$ and $q$ with $|p|-|q|=n$ where $|p|$ denotes the length of a path $p,$ then it is the $n$-component of $L_K(E)$ producing a natural grading of $L_K(E)$ by the group of integers $\Zset.$ One can also grade $L_K(E)$  by any group $\Gamma$ as follows. Any function $w\colon E^1\to \Gamma,$ called the {\em weight} function, extends by $w(e^*)=w(e)^{-1}$ for $e \in E^1$ and $w(v)=\varepsilon$ for $v\in E^0,$ and, ultimately, by $w(pq^*)=w(p)w(q)^{-1}$ for any generator $pq^*$ of $L_K(E)$ (see \cite[Section 1.6]{Roozbeh_book}). Thus, $L_K(E)$ becomes $\Gamma$-graded with $L_K(E)_\gamma$ being the $K$-linear span of the elements $pq^*$ with weight $\gamma.$ 

\subsection{The graph monoid and the Grothendieck group of a graph}\label{subsection_graph_group_nongraded_case}

If $E$ is a graph, the graph monoid $M_E$ was defined for row-finite graphs in \cite{Ara_Moreno_Pardo} and for arbitrary graphs in \cite{Ara_Goodearl}. We briefly review this definition. 

Any edge $e\in E^1$ is a partial isometry of $ee^*$ and $\ra(e)=e^*e$ so that $[ee^*]$ and $[\ra(e)]$ are the same element in $\V(L_K(E)).$ Hence, the relation below  holds in $\V(L_K(E))$ by the (CK2)-axiom if $v$ is regular. 
\begin{equation}
[v]=\sum_{e\in \so^{-1}(v)}[\ra(e)]\label{nepotrebno1}\tag{1}
\end{equation}

For any infinite emitter $v$ and any finite and nonempty $Z\subseteq \so^{-1}(v),$ one considers the element $q_Z$ representing $v-\sum_{e\in Z}ee^*.$ We refer to the elements of the form $q_Z$ as the {\em improper vertices} (and we note that this term was not used before).  When we need to emphasize that $q_Z$ is related to the infinite emitter $v$ (in the sense that $Z\subseteq \so^{-1}(v)$) we write $q_Z^v$ for $q_Z.$ Also, whenever the notation $q_Z$ appears, it is to be understood that there is an infinite emitter $v$ and that $Z$ is a finite and nonempty subset of $\so^{-1}(v).$
For any finite sets $Z$ and $W$ such that $\emptyset\subsetneq Z\subsetneq W \subsetneq \so^{-1}(v),$ it is direct to check that the relations

\begin{equation}
[v]=[q_Z]+\sum_{e\in Z}[\ra(e)]\;\;\mbox{ and }\;\;[q_Z]=[q_W]+\sum_{e\in W-Z}[\ra(e)]\label{nepotrebno2}\tag{2 and 3}
\end{equation}
also hold in $\V(L_K(E)).$
So, one aims to define $M_E$ so that the relations (1), (2 and 3) are the {\em only} relations which hold in $M_E.$ This is achieved in the following way. 

Let $F_E$ be a free commutative monoid generated by the elements indexed by the proper and improper vertices of $E.$ To be consistent with \cite{Ara_Goodearl}, \cite{Ara_et_al_Steinberg} and \cite{Talented_monoid}, we abuse the notation and refer to the generator indexed by a proper vertex $v\in E^0$ as $v$ and, similarly, to the generator indexed by $q_Z$ by $q_Z$. 
The monoid $M_E,$ called the {\em graph monoid}, is the quotient of $F_E$ with respect to the the congruence closure $\sim$ of the relation
$\to_1$ defined on $F_E-\{0\}$ by 
\[a+v \to_1 a+ \sum_{e\in s^{-1}(v)}r(e),\]
whenever $v$ is a regular vertex and $a\in F_E$ and by 
\[a+ v\to_1 a+ q_Z+\sum_{e\in Z}\ra(e)\;\;\mbox{ and }\;\;a+q_Z\to_1 a+ q_W+\sum_{e\in W-Z}\ra(e)\]
whenever $v$ is an infinite emitter and $Z$ and $W$ are finite and such that $\emptyset\subsetneq Z \subsetneq W\subsetneq \so^{-1}(v).$ 

One often considers an intermediate step of this construction and lets 
$\to$ be the transitive and reflexive closure of $\to_1$
on $F_E$ so that $\to$ is a pre-order. In this case, $\sim$ is the congruence on $F_E$ generated by the relation $\to$ (i.e. the symmetric closure of the pre-order $\to$). 

We use the notation $[v]$ for the congruence class of $v$ as an element of $M_E.$ As a side note, we add that the map $[v]\mapsto [vL_K(E)]$ extends to a pre-ordered monoid isomorphism of $M_E$ and $\V(L_K(E))$ (here $\V(L_K(E))$ is given using the finitely generated projective right modules) by \cite[Corollary 3.2.11]{LPA_book} (or \cite[Theorem 4.3]{Ara_Goodearl}). So, the Grothendieck group completion $G_E$ of $M_E$ is isomorphic to $K_0(L_K(E)).$

\subsection{The graph \texorpdfstring{$\Gamma$}{TEXT}-monoid and the Grothendieck \texorpdfstring{$\Gamma$}{TEXT}-group of a graph}\label{subsection_graph_group}

Let $\Gamma$ be a group and $w\colon E^1\to \Gamma$ be a function which we refer to as a weight determining a $\Gamma$-grading of $L_K(E).$ The following relations hold in the $\Gamma$-monoid $\V^{\Gamma}(L_K(E)).$
For every regular vertex $v,$ 
\[\gamma [v]=\sum_{e\in \so^{-1}(v)}\gamma w(e)[\ra(e)],\]
and for every infinite emitter $v$ and finite $Z$ and $W$ such that $\emptyset\subsetneq Z\subsetneq W\subsetneq \so^{-1}(v),$
\[\gamma [v]=\gamma [q_Z]+\sum_{e\in Z}\gamma w(e)[\ra(e)]\;\;\mbox{ and }\;\;\gamma [q_Z]=\gamma [q_W]+\sum_{e\in W-Z}\gamma w(e)[\ra(e)].\]

To adapt the original construction of $M_E$ to this setting, the authors of \cite{Ara_et_al_Steinberg} replaced generators $v$ and $q_Z$ of $F_E$ by $v(\gamma)$ and $q_Z(\gamma)$ for any $\gamma\in\Gamma$ and considered a free commutative monoid $F_E^\Gamma$ with the action of $\Gamma$ given by $\delta v(\gamma)=v(\delta\gamma)$ and $\delta q_Z(\gamma)=q_Z(\delta\gamma)$ for all $\gamma,\delta\in\Gamma.$ 
Then $M^\Gamma_E$ is the quotient of $F_E^\Gamma$ subject to the congruence closure $\sim$ of relation $\to_1$ defined just as in the previous section but with the three relations modified accordingly so that 
\[a+\gamma v \to_1 a+ \sum_{e\in s^{-1}(v)}\gamma w(e)r(e),\]
whenever $v$ is a regular vertex and $a\in F_E$ and by 
\[a+ \gamma v\to_1 a+ \gamma q_Z+\sum_{e\in Z}\gamma w(e)\ra(e)\;\;\mbox{ and }\;\;a+\gamma q_Z\to_1 a+\gamma q_W+\sum_{e\in W-Z}\gamma w(e)\ra(e)\]
whenever $v$ is an infinite emitter and $Z$ and $W$ are finite and such that $\emptyset\subsetneq Z\subsetneq W\subsetneq \so^{-1}(v)$. 

One downside of this approach is that $M^{\Gamma}_E$ is still considered to be a commutative monoid, not a commutative $\Gamma$-monoid. For example, if $E$ is a single vertex, $M^\Gamma_E$ is a direct sum of $|\Gamma|$-many copies of $\Zset^+$ (with a natural action of $\Gamma$) instead of being a single copy of $\Zset^+[\Gamma].$ Also, the abundance of generators can make some proofs less direct. Because of this, we adopt a simpler and more intuitive approach here: we let $M^{\Gamma}_E$ be defined by the same set of generators as when the weight function is trivial, but we let $F_E^\Gamma$ be a free commutative $\Gamma$-monoid, not a free commutative monoid. In this case, if $E$ is a single vertex, then $M_E$ is a single copy of $\Zset^+$ and $M^\Gamma_E$ is a single copy of $\Zset^+[\Gamma].$ The equivalence of ours and the construction from \cite{Ara_et_al_Steinberg} can be seen considering the graph covering $\ol E$ of $E$. 

So, we let $F_E^\Gamma$ be a free commutative $\Gamma$-monoid generated by proper and improper vertices. A nonzero element $a$ of $F_E^\Gamma$ has 
a {\em representation}, unique up to a permutation, as $\sum_{j=1}^n \alpha_ig_i$, where $g_i$ are {\em different} generators of $F_E^\Gamma$ and $\alpha_i\in \Zset^+[\Gamma]$. The {\em support} $\supp(a)$ of $a$ is the set $\{g_i\mid i=1,\ldots, n\}.$   

Let $k_\gamma\in \Zset^+$ be the coefficient of $\gamma\in \Gamma$ in $\alpha_i\in \Zset^+[\Gamma]$ in the above representation. By writing each $k_\gamma>0$ as the sum $1+1+\ldots+1,$ one obtains the format 
$a=\sum_{j=1}^m \gamma_j g_j$ for some positive integer $m$ and $\gamma_j\in \Gamma, j=1,\ldots, m.$ We allow the generators $g_j$ and $g_k$ to be possibly equal for $j\neq k$ in this form, also unique up to a permutation. We refer to it as a {\em normal representation} of $a$ and we say that each summand $\gamma_j g_j$ of this representation is a {\em monomial} of $a$. We can still write $\supp(a)=\{g_j \mid  j=1,\ldots, m\}$ because any possible repetition of an element does not impact $\supp(a)$ as a set.  

For example, if $\Gamma$ is the infinite cyclic group generated by $x,$ $v$ is a vertex of $E,$ and $a=xv+3v,$ then $(x+3)v$ is a representation of $a$ and $xv+v+v+v$ is a normal representation of $a.$

To shorten some statements, we say that a vertex $v,$ considered as a generator of $F_E^\Gamma,$ is {\em regular} if $v$ is regular as a vertex of $E.$ We also say that a generator $v\in F_E^\Gamma$ is a sink or an infinite emitter, if $v$ is a sink or an infinite emitter as a vertex of $E$. An element $a\in F_E^\Gamma$ is {\em regular} if every element of $\supp(a)$ is regular.

We define the {\em graph $\Gamma$-monoid} $M^\Gamma_E$ as a quotient of $F^\Gamma_E$ subject to the congruence closure $\sim$ of the relation $\to_1$ on $F_E^\Gamma-\{0\}$ defined by (A1), (A2) and (A3) below 
for any $\gamma\in \Gamma$ and $a\in F_E^\Gamma.$  
\begin{enumerate}
\item[(A1)] If $v$ is a regular vertex, then  
\[a+\gamma v\to_1 a+\sum_{e\in s^{-1}(v)}\gamma w(e)\ra(e).\]
\item[(A2)] If $v$ is an infinite emitter and $Z$ a finite and nonempty subset of $\so^{-1}(v),$ then  
\[a+\gamma v\to_1 a+\gamma q_Z+\sum_{e\in Z}\gamma w(e)\ra(e).\]
\item[(A3)] If $v$ is an infinite emitter and $Z\subsetneq W$ are finite and nonempty subsets of $\so^{-1}(v),$ then 
\[a+\gamma q_Z\to_1 a+\gamma q_W+\sum_{e\in W-Z}\gamma w(e)\ra(e).\]
\end{enumerate}

So, if $\to$ is the reflexive and transitive closure of $\to_1$ 
on $F_E^\Gamma,$ then $\sim$ is the congruence on $F^\Gamma_E$ generated by the relation $\to$. This means that the relation $a\sim b$ holds for some $a,b\in F_E^\Gamma-\{0\}$ if and only if there is a nonnegative integer $n$ and $a=a_0,\ldots,a_n=b\in F_E^\Gamma-\{0\}$ such that $a_i\to_1 a_{i+1}$ or $a_{i+1}\to_1a_i$ for all $i=0,\ldots,n-1.$ We refer to such $n$ as the {\em length of the sequence} $a_0, \ldots, a_n$ and we write $a\sim^n b$ to emphasize the length. In particular, if $a\to b,$ the sequence can be chosen so that $a_i \to_1 a_{i+1}$ for all $i=0,\ldots, n-1.$ In this case, we write $$a\to^n b.$$ Note that $a\to^1 b$ is just $a\to_1 b$ and that $a\to^0b$ is just $a=b.$ 

To shorten the notation in multiple proofs, if $g$ is a generator of $F_E^\Gamma,$ and one of the three axioms is applied to $g,$ we use $\ra(g)$ to denote the resulting term on the right side of relation $\to_1:$ $\sum_{e\in \so^{-1}(v)}w(e)\ra(e)$ if $g=v$ is a regular vertex, $q_Z+\sum_{e\in Z}w(e)\ra(e)$ for some finite and nonempty subset $Z$ of $\so^{-1}(v)$ if $g=v$ is an infinite emitter, or $q_W+\sum_{e\in W-Z}w(e)\ra(e)$ for some finite $Z$ and $W$ such that $\emptyset\subsetneq Z\subsetneq W\subsetneq \so^{-1}(v)$ if $g=q^v_Z$ for an infinite emitter $v.$ The element $\ra(g)$ is uniquely determined just for (A1). However, for a {\em fixed} use of (A2) or (A3) which is not changed within a proof, the notation $\ra(g)$ is a well-defined shortening. Such uniform treatment enables us to condense some proofs by avoiding considerations of three separate cases depending on which axiom is used.   

Another benefit of our approach is that the proofs of many known statements in the case when $\Gamma$ is trivial directly transfer to the case when $\Gamma$ is not trivial. For example, if $[g]$ denotes the congruence class of a generator $g$ of $F_E^\Gamma,$ 
the map $[g]\mapsto [gL_K(E)]$ extends to a pre-ordered $\Gamma$-monoid isomorphism of $M^\Gamma_E$ and $\V^\Gamma(L_K(E))$ and the proof of 
the case when $\Gamma$ is trivial (see, for example, \cite[Theorem 4.3]{Ara_Goodearl}) directly adapts to the case when $\Gamma$ is arbitrary. 
In \cite[Proposition 5.7]{Ara_et_al_Steinberg}, this monoid isomorphism is shown to exist by considering the graph covering. 

Lemma \ref{confluence} greatly simplifies many proofs which involve handling relation $\sim.$ 
Parts of this lemma can be shown by directly generalizing the proofs of \cite[Lemmas 5.6 and 5.8]{Ara_Goodearl}. We add some new elements in part (1) of Lemma \ref{confluence} to control the length of sequences for certain relations. We also note that part (2), the Confluence Lemma, is shown for general $\Gamma$ in \cite[Lemma 5.9]{Ara_et_al_Steinberg} but using the graph covering. The Confluence Lemma is key for showing that the monoid $M_E^\Gamma$ has the refinement property (see \cite[Proposition 5.9]{Ara_Goodearl}).  

\begin{lemma}
Let  $E$ be a graph, $\Gamma$ a group, $w\colon E^1\to\Gamma$ a weight function and $a,b\in  F^\Gamma_E-\{0\}$. 
\begin{enumerate}[\upshape(1)]
\item (The Refinement Lemma) If $a=a'+a''$ for some $a', a''\in F^\Gamma_E$ and if $a\to^n b$, then $b$ has summands $b', b''\in F^\Gamma_E$ and $n$ has summands  $i,j$ such that  $b=b'+b'',$ $i+j=n,$  $a'\to^i b',$ and $a''\to^j b''.$ 
\item (The Confluence Lemma) The relation $a\sim b$ holds if and only if $a \to c$ and $b\to c$ for some $c \in F_E-\{0\}.$  
\end{enumerate}
\label{confluence}
\end{lemma}
\begin{proof}
We show (1) by induction on $n.$ If $n=0$ then $a=b$ and we can take $b'=a'=a=b,$ $b''=a''=0,$ and $i=j=0.$ Assuming the induction hypothesis, let $a=a_0\to_1 a_1\to_1\ldots \to_1 a_n=b$ and let $\gamma g$ be a monomial of $a$ so that $a_1$ is obtained by replacing $\gamma g$ by $\gamma\ra(g).$  Since $a=a'+a'',$ $\gamma g$ is a summand of either $a'$ or $a''.$ Say it is $a'$ (the case when it is $a''$ is analogous) and let $a'=c+\gamma g$ for some $c\in F_E^\Gamma.$ For $a_1'=c+\gamma \ra(g)$ and $a_1''=a'',$ $a'\to^1 a_1'$ and $a''\to^0 a_1''.$ The induction hypothesis implies the existence of $b', b''\in F^\Gamma_E$ and $i,j$ such that such that $b=b'+b'',$ $i+j=n-1,$ $a_1'\to^i b'$ and $a_1''\to^j b''.$ Thus, $i+1+j=n,$ $a'\to^1 a_1'\to^i b',$ and $a''\to^0 a_1''\to^j b''$ and so $a'\to^{i+1}b'$ and $a''\to^j b''.$  
 
The direction $\Leftarrow$ of (2) is direct since if $a\to c$ and $b\to c,$ then $a\sim c$ and $b\sim c$ so that $a\sim b.$ First, we show the direction $\Rightarrow$ of (2) for finite graphs using induction on $n$ for $a\sim^n b.$ If $n=0,$ $a=b$ and we can take $c=a=b.$ Assuming the induction hypothesis, let $a\sim^n b,$ $a_0=a, a_n=b$ and let $a_i\to_1 a_{i+1}$ or  $a_{i+1}\to_1 a_i$ for some $a_i\in F_E^\Gamma$ for $i=0,\ldots, n-1.$ Since $a_1\sim^{n-1} b,$ there is $d$ such that $a_1\to d$ and $b\to d.$ Then either $a\to_1 a_1$ or $a_1\to_1 a.$ In the first case, we can take $c=d.$  In the second case, there is a monomial $\gamma g$ of $a_1$ so that $a_1=a'+\gamma g$ for some $a'$ and $a=a'+\gamma \ra(g).$ By part (1), 
$d=d'+d''$ for some $d'$ and $d''$ such that $a'\to d'$ and $\gamma g\to^l d''$ for some $l\geq 0$. If $l=0,$ then $d''=\gamma g$ so $d=d'+\gamma g.$ Let $c=d'+\gamma \ra(g).$ Then we have that $a=a'+\gamma\ra(g)\to d'+\gamma\ra(g)=c$ and $b\to d=d'+\gamma g\to_1 d'+\gamma \ra(g)=c.$ 

If $l$ is positive, we use the assumption that $E$ is finite to conclude that there are no infinite emitters so that $g$ is necessarily a regular vertex and $a_1\to_1 a$ is an application of (A1.) Hence, the relation $\gamma g\to d''$ necessarily decomposes as $\gamma g\to_1 \gamma\ra(g)\to d''$ and we have that $a_1=a'+\gamma g\to_1 a=a'+\gamma\ra(g)\to d'+d''=d.$ So, in this case we can also take $c=d.$ 

To complete the proof in the case when $E$ is an arbitrary graph, we use the argument of the proof of \cite[Lemma 5.9]{Ara_et_al_Steinberg} relying on \cite[Construction 5.3]{Ara_Goodearl}. If $R(E)$ denotes the set of regular vertices of $E$, the pair $(E, R(E)),$ considered as an element of an appropriate category from \cite[Section 3]{Ara_Goodearl}, can be represented as a direct limit of pairs $(E',S)$ where $E'$ is a finite subgraph of $E$ and $S$ is a subset of $R(E')$ (see \cite[Proposition 3.3]{Ara_Goodearl} for details). The pair $(E',S)$ gives rise to the relative graph $E'_S$ of $E'$ with respect to $S$ such that the bijection on the generators of the corresponding free $\Gamma$-monoids produces a natural $\Gamma$-monoid isomorphism (see \cite[Theorem 3.7]{Muhly_Tomforde} and the graded version in \cite[Lemma 2.2]{Lia_no-exit}). Hence, if $a,b\in F_E^\Gamma$ correspond to elements $a'$ and $b'$ of $F^\Gamma_{E'_S}$ for some finite subgraph $E'$ and some subset $S$ of $R(E'),$ then the relation $a\sim b$ holds in $F_E^\Gamma$ if and only if $a'\sim b'$ holds in $F_{E'_S}^\Gamma.$ Assuming that $a\sim b$ holds, we have that $a'\sim b'$ holds. By the proven claim for finite graphs, there is $c'\in F_{E'_S}^\Gamma$ such that the relations $a'\to c'$ and $b'\to c'$ hold in $F^\Gamma_{E'_S}.$ If $c\in F_E^\Gamma$ corresponds to $c',$ these relations imply that $a\to c$ and $b\to c$ hold in $F_E^\Gamma.$ 
\end{proof}

One can also show the Confluence Lemma directly, by considering an arbitrary graph $E$ and discussing possibilities that the relation $a_1\to_1 a$ in the above proof is obtained by (A2) or (A3).     

We conclude this section by a remark: the Graded Classification Conjecture is false if the pre-ordered $\Gamma$-monoids (equivalently $\Gamma$-groups) of the graphs are replaced by the free $\Gamma$-monoids. Indeed, let $E$ and $F$ be the graphs below and $\Gamma$ be the group of integers. 
$$
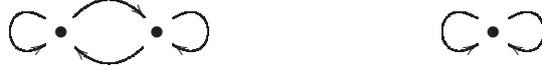
\xymatrix{ {\bullet}\ar@(ul,dl) \ar@/^1pc/[r]  & {\bullet} \ar@/^1pc/[l]  \ar@(ur,dr)}\hskip4cm \xymatrix{ {\bullet}\ar@(ul,dl) \ar@(ur,dr)}$$

The graph $E$ is an out-split of the graph $F$, so the Leavitt path algebras of $E$ and $F$ are graded isomorphic (see \cite[Theorem 2.8]{Abrams_et_al_classification}). Hence, $M_E^\Gamma$ and $M_F^\Gamma$ are isomorphic and so are $G_E^\Gamma$ and $G_F^\Gamma.$ Alternatively, one can show the existence of these isomorphisms by noting that $M_E^\Gamma$ and $M_F^\Gamma$ are both isomorphic to $\Zset^+[\frac{1}{2}]$ and, consequently, $G_E^\Gamma$ and $G_F^\Gamma$ are both isomorphic to $\Zset[\frac{1}{2}]$. However, $F_E^\Gamma$ and $F_F^\Gamma$ are not isomorphic as $\Gamma$-monoids since one has two while the other has one generator. This example illustrates that the $\Gamma$-monoid $F_E^\Gamma$ of a graph $E$ is informative only when considered together with the relation $\sim.$

\section{Connectivity}\label{section_connectivity}

In this section and the rest of the paper, $\Gamma=\{x^n \mid  n\in \Zset\}$ is the infinite cyclic group with generator $x$ and $E$ is an arbitrary graph. To simplify the terminology in some of the proofs, we say that $n$ is the {\em degree} of the monomial $x^{n}g$ where $g$ is a generator of $F_E^\Gamma.$  First, we characterize the relation $\to$ in terms of the graph-theoretic properties (Proposition \ref{connecting}). 

If $v$ and $w$ are vertices of $E$ and $p$ a path from $v$ to $w,$ one can apply (A1) or (A2) to the vertices on $p$ to obtain that $v\to x^{|p|}w+a$ for some  $a\in F_E^\Gamma.$ Indeed, if $p$ is trivial, then $v=w$ and one can take $a=0.$ If $p=e_1e_2\ldots e_n,$ one can apply (A1) if $v$ is regular and (A2) if it is not, and then apply (A1) to $\ra(e_1)$ if it is regular and (A2) if it is not. Continuing this process, one obtains a sequence for \[v\to x^{|p|}w+a\] for some $a\in F_E^\Gamma,$ where the ``change'' $a$ reflects the existence of bifurcations from $p.$ For example, in the graph below
with $p=f,$ we have that $v\to xw+xu$ so $a=xu.$ 
$$\xymatrix{{\bullet}^{u} & {\bullet}^{v} \ar[r]^f\ar[l]_e & {\bullet}^{w}}$$

We generalize this process to improper vertices also. The terminology introduced below allows uniform treatment of generators of $F_E^\Gamma$ of both types and enables us to express the comparability properties in terms of the properties of the graph $E$.

\begin{definition} Let $g$ and $h$ be generators of $F_E^\Gamma.$ 
We say that {\em $g$ connects to $h$ by a path $p$}  (written $g\cto^p h$) if one of the following conditions hold.
\begin{enumerate}[(i)]
\item  $g=v$ and $h=w$ are proper vertices and $p$ is a path from $v$ to $w.$ In this case, $v\to x^{|p|}w+a$ holds for some $a\in F_E^\Gamma$ as we pointed out above.  
 
\item $g=v$ is a proper vertex,  $h=q^w_Z$ for an infinite emitter $w$ and some $Z,$ and $p$ is a path from $v$ to $w.$ In this case, $v\to x^{|p|}w+a'\to x^{|p|}q_Z+a$ for some $a'\in F_E^\Gamma$ and $a=a'+\sum_{e\in Z}x^{|p|+1}\ra(e).$ Note that if $v=w$ and $p$ is trivial then $\to$ can be chosen to be a single application of (A2). If $v=w$ and $p$ has positive length, then $v$ is necessarily on a cycle. 

\item $g=q^v_Z$ for an infinite emitter $v$ and some $Z,$ $h=w$ is a proper vertex, and $p=eq$ is a path from $v$ to $w$ such that $e\notin Z.$ In this case, $$q_Z\to q_{Z\cup\{e\}}+x\ra(e)\to  q_{Z\cup\{e\}}+x^{|p|}w+a'=x^{|p|}w+a$$ 
for some $a'\in F_E^\Gamma$ and for $a=a'+ q_{Z\cup\{e\}}.$ If $v=w,$ then $v$ is on a cycle.

\item $g=q^v_Z$ for some $v$ and $Z,$ $h=q^w_W$ for some $w$ and $W,$ $p$ is a path from $v$ to $w,$ and one of the following two scenarios hold. 
\begin{itemize}
\item If $p$ is trivial, then $v=w$ and $Z\subseteq W.$ If $Z=W,$ then $q_Z\to x^0q_Z$ and if $Z\subsetneq W$ and $a=\sum_{e\in W-Z}x\ra(e),$  then
\[q_Z\to x^0q_W+\sum_{e\in W-Z}x\ra(e)=x^0q_W+a.\]

\item If $p$ has positive length, then $p=eq$ for some $e\notin Z.$ In this case, $$q_Z\to q_{Z\cup\{e\}}+x\ra(e)\to  q_{Z\cup\{e\}}+x^{|p|}w+a'\to  q_{Z\cup\{e\}}+x^{|p|}q_W+\sum_{f\in W}x^{|p|+1}\ra(f)+a'=x^{|p|}q_W+a$$
for some $a'\in F_E^\Gamma$ and $a=a'+q_{Z\cup\{e\}}+\sum_{f\in W}x^{|p|+1}\ra(f).$ If $v=w,$ then $v$ is on a cycle.  
\end{itemize}
\end{enumerate}
\label{definition_connecting}
\end{definition}

Definition \ref{definition_connecting} enables us to deal with every generator of $F_E^\Gamma$ in a uniform way. In particular, in any of the above four cases, we have that \[g\to x^{|p|}h+a\]
for some element $a\in F_E^\Gamma$ and a path $p.$ In this case, we say that {\em $h$ is obtained from $g$ following the path $p$}. The element $a$ reflects the existence of bifurcations from $p$. In Corollary \ref{converse_connecting}, we show the converse: $g\to x^{n}h+a$ implies that $g\cto^p h$ for a path $p$ of length $n$.  

We say that $g$ {\em connects to} $h$, written $g\cto h,$ if there is a path $p$ such that $g\cto^p h.$ If $v$ and $w$ are vertices, $v\cto w$ is usually written $v\geq w$ (see \cite[Definition 2.0.4]{LPA_book}). However, we reserve the relation $\geq$ for the order on the monoid $M_E^\Gamma.$ 
It is direct to check that $\cto$ is reflexive and transitive. 

Note that a proper vertex $v$ is on a cycle if and only if $v$ connects to $v$ by a path of positive length. Definition \ref{definition_connecting} enables us to talk about improper vertices being on cycles: we say that any generator $g$ of $F_E^\Gamma$ {\em is on a cycle} if $g$ connects to $g$ by a path of positive length. We say that $g$ {\em is on an exit from a cycle $c$} if $g$ is not on $c$ and there is a generator $h$ of $F_E^\Gamma$ which is on $c$ such that $h$ connects to $g.$ By Definition \ref{definition_connecting}, $q^v_Z$ is on a cycle if and only if there is $e\in \so^{-1}(v)-Z$ and a path $p$ with $\ra(p)=v, \so(p)=\ra(e)$ such that $ep$ is a cycle.  

If $a\to b$ and $a=\sum_{i=1}^k x^{m_i}g_i$ and $b=\sum_{j=1}^l x^{t_j}h_j$ are normal representations of $a$ and $b$ respectively, repeated use of  
the Refinement Lemma \ref{confluence}(1) ensures the existence of a partition $\{I_1, \ldots, I_k\}$ of $\{1, \ldots, l\}$ and summands $b_i$ of $b$ such that $b=\sum_{i=1}^k b_i,$ $b_i=\sum_{j\in I_i}x^{t_j}h_j,$ and $x^{m_i}g_i\to b_i.$ 
This implies that $t_j\geq m_i$ for all $j\in I_i.$ Proposition \ref{connecting} implies the existence of paths 
$p_{ij}$ with $|p_{ij}|=t_j-m_i$ and $g_i\cto^{p_{ij}} h_j$ in this case. We introduce this idea of partitioning $b$ according to $a$ if $a\to b$ in  Proposition \ref{connecting} and use it again in Section \ref{subsection_partition}. Proposition \ref{connecting} describes the relation $a\to b$ in terms of the properties of the generators in the supports of $a$ and $b$ and the length of the paths connecting them.   

\begin{proposition}
Let $a,b\in F_E^\Gamma-\{0\}$ and $a=\sum_{i=1}^k x^{m_i}g_i$ and $b=\sum_{j=1}^l x^{t_j}h_{j}$ be normal representations of $a$ and $b$ respectively. The following conditions are equivalent.
\begin{enumerate}[\upshape(1)]
\item The relation $a\to b$ holds. 
 
\item There is a partition $\{I_1, \ldots, I_k\}$ of $\{1, \ldots, l\}$ and finitely many paths $p_{ij},  j\in I_i, i=1,\ldots, k,$ such that $g_i\cto^{p_{ij}}h_j,$  $|p_{ij}|=t_j-m_i$ for all $j\in I_i, i=1,\ldots, k,$ and  
$$b=\sum_{j=1}^l x^{t_j}h_{j}=\sum_{i=1}^k \sum_{j\in I_i}x^{m_i+|p_{ij}|}h_{j}.$$ If $p$ is a prefix of $p_{ij}$ and $v=\ra(p),$ let
\[P_p=\{e\in \so^{-1}(v)\mid e \mbox{ is on } p_{ij'}\mbox{ for some }j'\in I_i\}.\] Then the following hold. 
\begin{enumerate}[\upshape(i)]
\item If $v$ is regular and $P_p$ nonempty, then $P_p=\so^{-1}(v).$

\item If $v$ is an infinite emitter and $P_p$ nonempty, then there is $j'\in I_i$ such that $h_{j'}=q_Z^v$ for some $Z$ such that $P_p\subseteq Z.$

\item The relation $t_j=|p|+m_i$ holds if and only if $p=p_{ij}$ and $h_j=q_Z^v$ for some $Z$ implies $P_p\subseteq Z.$
\end{enumerate}
\end{enumerate}
\label{connecting} 
\end{proposition}

Before presenting the proof, let us motivate it by some examples. 

\begin{example}
\begin{enumerate}
\item In the graph below,   $u\to xv$ and $w\to xv$ so $u+w\to xv+xv.$ For this last relation, $k=2, l=2$ and one can take $I_1=\{1\}, I_2=\{2\},$ $p_{11}=e,$ and $p_{22}=f$ so condition (2) holds.
$$\xymatrix{{\bullet}^{u} \ar[r]^e & {\bullet}^{v} & {\bullet}^{w}\ar[l]_f}$$
By condition (2) also, $u\to x^2v+a$ fails for any $a$ since there is no path of length 2 from $u$ to $v.$ 
 
\item In the graph below, $v\to xu+xw.$ For this relation, $k=1, l=2$ and one can take $I_1=\{1,2\},$ $p_{11}=e,$ and $p_{12}=f$ so condition (2) holds. 
$$\xymatrix{{\bullet}^{u} & {\bullet}^{v} \ar[r]^f\ar[l]_e & {\bullet}^{w}}$$
Although $v$ connects to $w$ by a path of length one, $v\to \alpha w$ fails for any $\alpha\in \Zset^+[\Gamma]$
since the path from $v$ to $w$ has a bifurcation towards $u$ so $u$ must appear in any ``result'' obtained following a path from $v$ to $w$ by condition (2)(i). 

\item The relation $v_0\to a$ fails for any $a$ with $\supp(a)$ consisting of sinks only in the graph below.  
$$
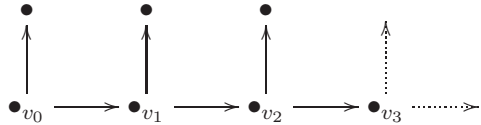
\xymatrix{   
\bullet     & \bullet     & \bullet  &   & \\   
\bullet_{v_0} \ar[r]\ar[u] & \bullet_{v_1} \ar[r]  \ar[u] & \bullet_{v_2} \ar[r]\ar[u] & \bullet_{v_3} \ar@{.>}[r] \ar@{.>}[u] &}$$
Indeed, all paths from $v_0$ to finitely many sinks have a bifurcation on a path which does not end in any sink. Hence, if $v_0\to a$ then $a$ necessarily has $v_i$ in its support for some $i\geq 0.$  

\item If $v\to^n a$ for $n>0$ in the graph below, 
$$
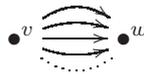
\xymatrix{{\bullet}^{v} \ar@{.} @/_1pc/ [r] _{\mbox{ } } \ar@/_/ [r] \ar [r] \ar@/^/ [r] \ar@/^1pc/ [r] & {\bullet}^{w}}$$
condition (2) implies the existence of an improper vertex in $\supp(a).$ Hence, $v\to \alpha w$ fails for any $\alpha\in  \Zset^+[\Gamma]$.
\end{enumerate}
\label{example_connecting}
\end{example}
 
\begin{proof}
Let us show direction $\Rightarrow$ by induction on $n$ for $a\to^nb.$ If $n=0,$ then $a=b$ so $k=l$ and one can permute the monomials in the normal representation of $b$ if necessary to get that $t_i=m_i$ for all $i=1,\ldots, k.$ In this case, one can take $I_i=\{i\}$ and $p_{ii}$ to be the trivial path which connects $g_i$ to $g_i$ for all $i=1,\ldots, k.$ In this case any prefix $p$ of $p_{ij}$ is trivial and relation $t_i=|p_{ij}|+m_i=|p|+m_i$ holds. Since $P_p=\emptyset,$ conditions (i) to (iii) hold.   

Considering the case $n=1$ shortens the arguments in the inductive step. If $n=1,$ reorder the terms of the normal representation of $a$ if necessary to assume that $b$ is obtained by applying an axiom to $x^{m_k}g_k$ and let $x^{m_k}\ra(g_k)$ denote the result of this application. Thus, $g_k$ is not a sink. Reorder the terms of the normal representation of $b$ to have that $b=\sum_{i=1}^{k-1} x^{m_i}g_i+x^{m_k}\ra(g_k)$ and let $x^{m_k}\ra(g_k)=\sum_{j\in J}x^{t_j}h_j$ for some finite subset $J$ of $\{1,\ldots,l\}.$  Let $I_i=\{i\}$ for $i=1,\ldots, k-1$ and $p_{ii}$ be the trivial path which connects $g_i$ and $g_i$ if $k>1.$ Let $I_k=J$ and $p_{kj}$ be the path (of length zero or one) which connects $g_k$ and $h_j$. Since $g_k$ is not a sink, there are just three possible cases, listed below, for $g_k$.  

\begin{enumerate}[1.]
\item $g_k$ is a regular vertex $v.$ In this case, $|J|=|\so^{-1}(v)|$ and we can label the elements of $\so^{-1}(v)$ such that $h_j=\ra(e_j)$ for $j\in J.$ Then $x^{t_j}h_j=x^{m_k+1}\ra(e_j)$ and so $t_j=m_k+1.$ Let $p_{kj}=e_j.$
If $p$ is a prefix of $e_j,$ then either $p=v$ in which case $t_j>|p|+m_i$ and $P_p=\so^{-1}(v),$ or $p=e_j$ in which case $t_j=|p|+m_i.$ In this case, if $\ra(e_j)$ is regular and $\ra(e_j)\neq v,$ then $P_p=\emptyset$ and if $\ra(e_j)=v,$ then $P_p=\so^{-1}(v).$   
 
\item $g_k$ is an infinite emitter $v.$ In this case, $\ra(g_k)=q_Z+\sum_{e\in Z}x\ra(e)$ for some $Z$ and $|J|=|Z|+1.$ We can label the elements of $Z$ such that $h_j=\ra(e_j)$ for $j\in J-{j_0}$ and $h_{j_0}=q_Z.$ Thus, $t_{j_0}=m_k$ and $t_j=m_k+1$ for $j\in J-\{j_0\}.$ Let $p_{kj_0}=v$ and $p_{kj}=e_j$ so that $|p_{kj}|=t_j-m_k$ for all $j\in J.$  
If $p$ is a prefix of $p_{kj_0}=v,$ then $p=v$, $t_j=|p|+m_i,$ $h_{j_0}=q_Z^v$ and $P_p=Z.$ 
If $p$ is a prefix of $e_j,$ then either $p=v$ or $p=e_j.$ In the first case, $t_j>|p|+m_i,$ and condition (ii) holds with $j=j_0.$ In the second case, $t_j=|p|+m_i,$ if $\ra(e_j)$ is regular, then $\ra(e_j)\neq v$ and so $P_p=\emptyset.$  
For $j\neq j_0,$ $h_j$ is a proper vertex and $h_{j_0}=q_Z$ with $P_p=Z.$ Thus, condition (iii) holds. 

\item $g_k$ is an improper vertex $q^v_Z.$ In this case, $\ra(g_k)=q_W+\sum_{e\in W-Z}x\ra(e)$ for some $W\supsetneq Z$ and $|J|=|W-Z|+1.$ We can label the elements of $W-Z$ such that $h_j=\ra(e_j)$ for $j\in J-{j_0}$ and $h_{j_0}=q_W.$ Thus, $t_{j_0}=m_k$ and $t_j=m_k+1$ for $j\in J-\{j_0\}.$ Let $p_{kj_0}=v$ and $p_{kj}=e_j$ so that $|p_{kj}|=t_j-m_k$ for all $j\in J.$  
If $p$ is a prefix of $p_{kj_0}=v,$ then $p=v$, $t_j=|p|+m_i,$ $h_{j_0}=q_W,$ and $P_p=W-Z\subseteq W.$ 
If $p$ is a prefix of $e_j,$ then either $p=v$ or $p=e_j.$ In the first case, $t_j>|p|+m_i,$ and condition (ii) holds with $j=j_0.$ In the second case, $t_j=|p|+m_i,$ if $\ra(e_j)$ is regular, then $\ra(e_j)\neq v$ and so $P_p=\emptyset.$ 
For $j\neq j_0,$ $h_j$  is a proper vertex and  $h_{j_0}=q_W$ with $P_p=W-Z\subseteq W.$  Thus, condition (iii) holds.
\end{enumerate}
By construction, $$b=\sum_{i=1}^{k-1} x^{m_i}g_i+\sum_{j\in J}x^{t_j}h_j =\sum_{i=1}^{k-1} x^{m_i+|p_{ii}|}g_i+\sum_{j\in I_k}x^{m_k+|p_{kj}|}h_j=
\sum_{i=1}^k \sum_{j\in I_i}x^{m_i+|p_{ij}|}h_{j}.$$
 
Assuming the induction hypothesis, let us consider a sequence $a_0=a\to_1 a_1\to_1\ldots\to_1 a_n=b.$ Let $a_{n-1}=\sum_{j'=1}^{l'} x^{t'_{j'}}h'_{j'}.$ By the induction hypothesis, there is a partition $\{I'_1, \ldots, I'_k\}$ of $\{1, \ldots, l'\}$ and finitely many paths $p_{ij'},  j'\in I'_i, i=1,\ldots, k,$ such that $g_i\cto^{p_{ij'}}h'_{j'},$ $|p_{ij'}|=t'_{j'}-m_i,$ and the required conditions hold for any prefix of $p_{ij'}$ for all $j'\in I'_i$ and $i=1,\ldots, k.$ The element $b$ is obtained from $a_{n-1}$ by application of one of the axioms to exactly one monomial $x^{t'_{j'}}h'_{j'}.$ Reordering the terms of $a_{n-1}$ if necessary, we can assume that it is the last one $x^{t'_{l'}}h'_{l'}.$ Reorder the terms of $b$ if necessary to have that $b=\sum_{j'=1}^{l'} x^{t'_{j'}}h'_{j'}+x^{t'_{l'}}\ra(h'_{l'})$ and let $x^{t'_{l'}}\ra(h'_{l'})=\sum_{j\in J}x^{t_j}h_j$ for some finite subset $J$ of $\{1,\ldots,l\}.$  
By construction, we have that $l=l'+|J|$ and that $l'$ is in $I'_{i_0}$ for exactly one $i_0.$ So we let
\[I_i=I'_i, \mbox{ if }i\neq i_0,\mbox{ and }I_{i_0}=J.\] 
If $i\neq i_0,$ for each $j\in I_i,$ $x^{t_j}h_j=x^{t'_{j'}}h'_{j'}$ for exactly one $j'\in I'_i.$ So, for such $j$ and $j',$ we let $p_{ij}=p_{ij'}$ so that $|p_{ij}|=|p_{ij'}|=t'_{j'}-m_i=t_j-m_i.$ 

For $i_0,$ we let $p_{i_0j}$ be the concatenation of $p_{i_0l'}$ and the path $p_{l'j}$
constructed as in the case $n=1$ for $h'_{l'}$ and $h_j$ for $j\in J=I_{i_0}.$ Since $g_{i_0}\cto^{p_{i_0l'}} h'_{l'}$ and $h'_{l'}\cto^{p_{l'j}}h_j$ for all $j\in J=I_{i_0},$ we have that $g_{i_0}\cto^{p_{i_0j}}h_j$ for all $j\in I_{i_0}.$
We have that $|p_{i_0l'}|=t'_{l'}-m_{i_0}$ and $|p_{l'j}|=t_j-t'_{l'}$ and so  
\[|p_{i_0j}|=|p_{i_0l'}|+|p_{l'j}|=t'_{l'}-m_{i_0}+t_j-t'_{l'}=t_j-m_{i_0}\;\;\mbox{ and }\] 
$$b=\sum_{j=1}^l x^{t_j}h_{j}=\sum_{j'=1}^{l'} x^{t'_{j'}}h'_{j'}+\sum_{j\in J}x^{t_j}h_j=
\sum_{i=1, i\neq i_0}^k \sum_{j\in I_i}x^{m_i+|p_{ij}|}h_{j}+ \sum_{j\in I_{i_0}}x^{m_{i_0}+|p_{i_0j}|}g_{i_0}=
\sum_{i=1}^k \sum_{j\in I_i}x^{m_i+|p_{ij}|}h_{j}.$$

If $p$ is a prefix of $p_{i_0j},$ then it is either a prefix of $p_{i_0l'}$ or $p=p_{i_0l'}q$ for some prefix $q$ of $p_{l'j}$ and one of the following three cases holds: first, $p$ is a proper prefix of  $p_{i_0l'},$ second, $q$ is a proper prefix of $p_{l'j}$ or, third, $q=p_{l'j}$ thus $p=p_{i_0j}.$ 
In the first case, $t'_{l'}>|p|+m_{i_0}$ and so $t_j\geq t'_{l'}>|p|+m_{i_0}.$ In the second case, $t_j>|q|+t'_{l'}=|q|+|p_{i_0l'}|+m_{i_0}=|p|+m_{i_0}.$ In the last case, $t_j=|p|+m_{i_0}$ and if $h_j=q_Z^v$ for some $Z$ then $P_p \subseteq Z$ since this condition holds for $a_{n-1}\to_1 b$ by the first induction step. In all three cases, if $\ra(p)$ is regular and $P_p\neq\emptyset,$ we can use induction hypothesis to conclude that $P_p=\so^{-1}(\ra(p))$ and, if $\ra(p)$ is an infinite emitter $v$ and $P_p\neq\emptyset,$  we can use induction hypothesis to conclude that there is $j'$ such that $h_{j'}=q_Z^v$ for some $Z$ such that $P_p\subseteq Z.$  Thus, in any case, conditions (i) to (iii) hold. 

Let us use induction on $k$ to show direction $\Leftarrow.$ If $k=1$ and $a=x^mg,$ let $p_j, j=1,\ldots, l$ denote the paths which exist by condition (2). We show the claim using induction on $n=\sum_{j=1}^l|p_j|.$ If this length is zero, then we claim that $b=a.$
Indeed, since $|p_j|=0,$ the relation $g\cto^{p_j}h_j$ implies that either $g=h_j,$ or $g=v$ for some infinite emitter $v$ and $h_j=q_Z^v,$ or that $g_i=q_W^v$ and $h_j=q_Z^v$ for some $v$ and $W\subsetneq Z.$ However, in the second and third case we would have that $P_v\subseteq Z$ by condition (2) so there would have to be some paths $p_{j'}$ of length at least one which cannot happen since $n=0$. Hence, $a=b$ and, thus, $a\to b.$  

Assuming the induction hypothesis, let $n=\sum_{j=1}^l|p_j|>0.$ Since $n>0,$ $a\neq b$ and there is $j=1,\ldots, l$  such that $|p_j|>0.$ If $p_j=e_0p$ for an edge $e_0$ and a path $p,$ let $v=\so(e_0).$ Since $e_0\in P_v,$ $P_v\neq \emptyset.$ We have exactly three possibilities for $g,$ listed below.   

\begin{enumerate}[1.]
\item $g=v$ is regular. Since $P_v$ is nonempty, $P_v=\so^{-1}(v)$ by (i). Let $a_1=\ra(v)=\sum_{e\in P_v}x^{m+1}\ra(e).$ Note that $a\to_1 a_1$ by (A1). We claim that condition (2) holds for $a_1$ and $b.$ 

For $e\in P_v=\so^{-1}(v),$ there is some $j=1,\ldots, l$ such that $p_j=eq_j$ for some path $q_j$ and so the set $$I_e=\{j\in \{1, \ldots, l\}\mid  e\mbox{ is the first edge of }p_j\}$$ is nonempty. 
Since the first edge of $p_j$ is in $P_v=\so^{-1}(v)$ for any $j,$ we have that $\bigcup_{e\in P_v} I_e=\{1, \ldots, l\}.$ If $j\in I_e\cap I_{e'},$ then $e=e'$ since the first edge of a path is unique.  
As $t_j=|p_j|+m,$ we have that $t_j=|q_j|+1+m$ and $$b=\sum_{j=1}^l x^{t_j}h_j=\sum_{j=1}^l x^{|p_j|+m}h_j=\sum_{j=1}^l x^{|q_j|+1+m}h_j.$$
If $q$ is a prefix of $q_j,$ then $eq$ is a prefix of $p_j$ and conditions (i) to (iii) hold for $q$ because they hold for $eq.$ Thus, we have that $a_1\to b$ by induction hypothesis. Since $a\to_1 a_1,$ we have that $a\to b.$

\item $g=v$ is an infinite emitter. In this case, let $a_1=x^m q_{P_v}+\sum_{e\in P_v}x^{m+1}\ra(e).$ So that $a\to_1 a_1$ by (A2). Since $P_v\neq\emptyset,$ there is $j$ such that $h_j=q_Z^v$ for some $Z$ with $P_v\subseteq Z$ by (ii). By (iii), such $j$ can be found so that $t_j=|p_j|+m.$
Reorder the terms of $b$ if necessary so that we can assume that $j=1.$ We check that condition (2) holds for $a_1$ and $b.$ 

For $e\in P_v$, there is $j=2,\ldots, l$ such that $p_j=eq_j$ for some path $q_j$ and so the sets $I_e, e\in P_v,$ defined as in the previous case, are nonempty and mutually disjoint. Let $I_1=\{1\}$ and $q_1=p_1.$ Since the first edge of $p_j$ is in $P_v$ for every $j=2,\ldots, l,$ $\bigcup_{e\in P_v} I_e=\{2, \ldots, l\},$ so $I_1\cup \bigcup_{e\in P_v} I_e=\{1, 2, \ldots, l\}.$ Hence, $\{I_1\}\cup\{I_e| e\in P_v\}$ is a partition of $\{1,\ldots,l\}.$
For $j=2,\ldots, l,$ $t_j=|p_j|+m=|q_j|+1+m,$ $t_1=|p_1|+m=|q_1|+m,$ and 
$$b=\sum_{j=1}^l x^{|p_j|+m}h_j=x^{|q_1|+m}h_1+\sum_{j=2}^l x^{|q_j|+1+m}h_j.$$
If $q$ is a prefix of $q_j$ for $j>1,$ then $e_iq$ is a prefix of $p_j$ and conditions (i) to (iii) hold for $q$ because they hold for $e_iq.$
If $q$ is a prefix of $q_1=p_1,$ then the requirements also hold. Thus, $a_1\to b$ by induction hypothesis. Since $a\to_1 a_1,$ we have that $a\to b.$

\item $g=q_Z^v$ for some $Z.$ In this case, $P_v$ is a proper superset of $Z.$ Let $a_1=x^m q_{P_v}+\sum_{e\in P_v-Z}x^{m+1}\ra(e)$ so that $a\to_1 a_1$ by (A3). By (ii), there is $j$ such that $h_j=q_W^v$ for some $W$ such that $P_v\subseteq W$ and, by (iii), such $j$ can be found so that $t_j=|p_j|+m.$ Reorder the terms of $b$ if necessary so that we can assume that $j=1.$ We check that condition (2) holds for $a_1$ and $b.$ 

For $e\in P_v-Z,$ there is some $j=2,\ldots, l$ such that $p_j=eq_j$ for some path $q_j$ and so the sets $I_e, e\in P_v-Z,$ defined as in the previous cases, are nonempty and mutually disjoint. If $I_1=\{1\}$ and $q_1=p_1,$ one shows that $\{I_1\}\cup\{I_e \mid e\in P_v-Z\}$ is a partition of $\{1,\ldots,l\}$ as in the previous case. Since $t_j=|p_j|+m=|q_j|+1+m$ for $j=2,\ldots, l$ and $t_1=|p_1|+m=|q_1|+m,$
$b=\sum_{j=1}^l x^{|p_j|+m}h_j=x^{|q_1|+m}h_1+\sum_{j=2}^l x^{|q_j|+1+m}h_j.$ The requirements on prefixes of $q_j$ can be checked just as in the previous case. Thus, we have that $a_1\to b$ by induction hypothesis. Since $a\to_1 a_1,$ we have that $a\to b.$
\end{enumerate}

This concludes the proof of the case $k=1$. Assuming the induction hypothesis, let us show the claim for $a$ with $k$ terms in its normal decomposition. Note that if condition (2) holds, then it holds for $a'=\sum_{i=1}^{k-1}x^{m_i}g_i$ and $b'=\sum_{i=1}^{k-1}\sum_{j\in I_i}x^{m_i+|p_{ij}|}h_j$ and for $x^{m_k}g_k$ and $\sum_{j\in I_k}x^{m_k+|p_{kj}|}h_j.$ By the induction hypothesis, we have that $a'\to b'$ and that $x^{m_k}g_k\to \sum_{j\in I_k}x^{m_k+|p_{kj}|}h_h.$ Hence, $a=a'+x^{m_k}g_k\to b=b'+\sum_{j\in I_k}x^{m_k+|p_{kj}|}h_j.$   
\end{proof}

We show two corollaries of Proposition \ref{connecting} which we use in Section \ref{subsection_stationary}. Recall that Definition \ref{definition_connecting} implies that $g\cto^ph$ implies $g\to x^{|p|}h+a$ for some $a.$ By the first corollary, the converse also holds. 

\begin{corollary} Let $g,h$ be generators of $F_E^\Gamma,$ $a\in F_E^\Gamma,$ and $m$ a nonnegative integer. Then $g\to x^mh+a$ holds if and only if there is a path $p$ of length $m$ such that $g\cto^ph.$ 
\label{converse_connecting} 
\end{corollary}
\begin{proof}
If $g\to x^mh+a$, condition (2) of Proposition \ref{connecting} holds by Proposition \ref{connecting}, so there is a path $p$ of length $m$ from $g$ to $h.$ The converse holds by Definition \ref{definition_connecting} (see the sentence following Definition \ref{definition_connecting}). 
\end{proof}

By the next corollary, if $a\to b,$ then each monomial of $b$ is obtained by a monomial of $a.$ This complements the Confluence Lemma. 

\begin{corollary}
If $g$ is a generator of $F_E^\Gamma,$ $a,b\in F_E^\Gamma,$ and $m$ an integer, then $a\to x^mg+b$ implies that there is $h\in\supp(a)$ and $k\leq m$ such that $x^kh$ is a monomial of $a$ and that $x^kh\to x^mg+c$ for some $c\in F_E^\Gamma.$   
\label{decomposing_the_initial_term_lemma}  
\end{corollary}
\begin{proof}
If $a\to x^mg+b$ holds, Proposition \ref{connecting} guarantees the existence of a monomial $x^kh$ of $a$ and a path $p$ such that $k+|p|=m$ and such that $h\cto^p g.$ Hence, $m-k=|p|\geq 0$ and $x^kh\to x^{k+|p|}g+c=x^mg+c$ for some $c\in F_E^\Gamma$ by Corollary \ref{converse_connecting}.
\end{proof}

\subsection{Connectivity of the supports}\label{subsection_connectivity_supports} Next, we associate the relation $a\to b$ to the properties of the supports of $a$ and $b$. 
The next definition is condition (2) of Proposition \ref{connecting} stripped down from any mention of degrees.

\begin{definition}
Let $G$ and $H$ be finite and nonempty sets of generators of $F_E^\Gamma.$ We write $G\to H$ if there are $k\geq |G|$ and $l\geq |H|$ such that the elements of $G$ and $H$ can be indexed as $g_1,\ldots, g_k$ and $h_1,\ldots, h_l$ (with repetitions allowed) respectively and there is a partition $\{I_1, \ldots, I_k\}$ of $\{1,\ldots, l\}$ and finitely many paths $p_{ij},  j\in I_i, i=1,\ldots, k,$ such that $g_i\cto^{p_{ij}}h_j$ for all $j\in I_i, i=1,\ldots, k$ and such that if $p$ is a prefix of $p_{ij}$ and $P_p$ is as in condition (2) of Proposition \ref{connecting} then (2)(i) and (2)(ii) of Proposition \ref{connecting} and condition (iii) below hold. 
\begin{enumerate}[\upshape(i)]
\item[(iii)] If $v$ is an infinite emitter and $h_j=q_Z^v$ for some $Z,$ then $P_{p_{ij}}\subseteq Z.$ 
\end{enumerate}
\label{definition_connecting_support}
\end{definition}

\begin{corollary}
\begin{enumerate}[\upshape(1)]
\item If $a,b\in F_E^\Gamma-\{0\},$ then $a\to b$ implies $\supp(a)\to\supp(b).$
 
\item Let $a,b\in F_E^\Gamma-\{0\}$ be such that $\supp(a)\to \supp(b).$ Then, there is $c\in F_E^\Gamma-\{0\}$ such that $\supp(c)\subseteq \supp(b)$ and that $a\to c.$ 
\end{enumerate}
\label{corollary_connecting} 
\end{corollary}
\begin{proof}

Part (1) directly follows from Proposition \ref{connecting}. To show part (2), assume that $a=\sum_{i=1}^{k} x^{m_i}g_i$ and $b=\sum_{j=1}^{l} x^{t_j}h_j$ be such that $\supp(a)\to \supp(b).$ Let $m,n$ be the cardinalities of $\supp(a)$ and $\supp(b)$ respectively and $k'\geq m, l'\geq n,$ $\{I_1, \ldots, I_{k'}\}$ and $p_{i'j'}, j'\in I_{i'}, i'=1,\ldots, k'$ be as in Definition \ref{definition_connecting_support} for $\supp(a)$ and $\supp(b)$. Then, we let  $$a'=\sum_{i'=1}^{k'} g_{i'},\;\; \;\;b'_{i'}=\sum_{j'\in I_{i'}}x^{|p_{i'j'}|}h_{j'}\mbox{ for }i'=1,\ldots, k',\;\;\mbox{ and }\;\;b'=\sum_{i'=1}^{k'}b'_{i'}.$$ By construction, $\supp(a')=\supp(a),$ $\supp(b')=\supp(b)$ and $g_{i'}\to b_{i'}$ so that $a'\to b'$ holds by Proposition \ref{connecting}. For any $i=1,\ldots, k,$ there is $i'=1,\ldots, k'$ such that $g_i=g_{i'}.$ For such $i,$ let 
$$c_i=\sum_{j'\in I_{i'}}x^{m_i+|p_{i'j'}|}h_{j'}\;\;\mbox{ and let }\;\;c=\sum_{i=1}^{k}c_i.$$ We have that $\supp(c)\subseteq \supp(b')=\supp(b)$ and $x^{m_i}g_i\to c_i$ so that $a=\sum_{i=1}^{k} x^{m_i}g_i\to c=\sum_{i=1}^{k}c_i.$ The relation $x^{m_i}g_i\to c_i$ also implies that $c_i\neq 0$ so $c\neq 0.$ 
\end{proof}

We note that the converse of part (1) of Corollary \ref{corollary_connecting} does not have to hold. Also, for an element $c$ as in part (2) of Corollary \ref{corollary_connecting}, the relation $b\sim c$ does not have to hold even if $\supp(c)=\supp(b)$. Indeed, in the graph below,  $v\to xw$ so $\{v\}\to \{w\}.$
$$\xymatrix{{\bullet}^{v} \ar[r] & {\bullet}^{w}}$$
However, for $a=v$ and $b=w,$ we have that $\supp(a)\to\supp(b)$ but $a\to b$ fails since there are no paths of length zero from $v$ to $w.$ If $c=xw,$ then $\supp(c)=\supp(b),$ but we do not have that $b=w\sim c=xw$ since $w$ is a sink and the relation $w\sim d$ for some $d$ implies $d=w$ or $d=x^{-1}v$ by the Confluence Lemma. Also, using Theorem \ref{periodic}, it is direct that $w\sim xw$ cannot hold since $w$ is not a periodic element.    

\subsection{Connecting using (A1) only}\label{subsection_strong_connecting}

To emphasize that $a\to b$ is such that only (A1) is used, we write $a\sto b.$ If $E$ is a row-finite graph, then $\sto$ is just the relation $\to.$ If $V$ is a finite and nonempty set of regular vertices and $W$ a finite and nonempty set of proper vertices such that $V\to W$, we write $V\sto W.$ Corollary \ref{corollary_connecting} implies the corollary below. 

\begin{corollary}
\begin{enumerate}[\upshape(1)]
\item Let $a,b\in F_E^\Gamma-\{0\}$ such that $a\sto b$ and that $\supp(a)$ consists of regular vertices. Then $\supp(a)\sto \supp(b).$  
 
\item Let $a,b\in F_E^\Gamma-\{0\}$ be such that $\supp(a)\sto \supp(b).$ Then, there is $c\in F_E^\Gamma-\{0\}$ such that $\supp(c)\subseteq \supp(b)$ and that $a\sto c.$
\end{enumerate}
\label{strong_connecting}
\end{corollary}
\begin{proof}
To show (1), assume that $a\sto b$ and that $\supp(a)$ consists of regular vertices. By Corollary \ref{corollary_connecting}, $\supp(a)\to \supp(b).$ Since only (A1) is used in $a\sto b$, $\supp(b)$ does not contain any improper vertices, so $\supp(a)\sto \supp(b)$ by definition of $\sto$ for sets of vertices.  

To show (2), let $\supp(a)\sto \supp(b).$ By Corollary \ref{corollary_connecting}, there is $c\in F_E^\Gamma-\{0\}$ such that $\supp(c)\subseteq \supp(b)$ and $a\to c.$ Since the support of $a$ consists of regular vertices and the support of $b,$ thus of $c$ as well, of proper vertices, only (A1) can be applied in a sequence for $a\to c$. Hence, $a\sto c.$ 
\end{proof}

\section{Characterization of comparable elements}\label{section_comparable}

\subsection{Cancellative property}\label{subsection_cancellative}
First, we show that the monoid $M_E^\Gamma$ is cancellative by a direct proof. This was shown in  \cite[Corollary 5.8]{Ara_et_al_Steinberg} using the graph covering. Note that $M_E^\Delta$ may not be cancellative for a group $\Delta\neq \Gamma.$ In particular, if $E$ is a graph with a cycle with an exit and $\Delta$ is trivial, then $M_E^\Delta$ is not cancellative by \cite[Lemma 5.5]{Ara_et_al_Steinberg}.  

\begin{proposition} The $\Gamma$-monoid $M_E^\Gamma$ is cancellative. 
\label{cancellative}
\end{proposition}
\begin{proof}
Assume that $a+c\sim b+d$ holds in $F_E^\Gamma$ for some $d\in F^\Gamma_E$ such that $c\sim d.$ So,  
we have that $a+c\sim b+d\sim b+c.$ We show that $a\sim b$ using induction on $n$ for $a+c\sim^n b+c.$
If $n=1,$ then either $a+c\to_1 b+c$ or $b+c\to_1 a+c.$ In the first case, there is a generator $g$ in the support of $a$ or $c$ such that
$b+c$ is obtained by replacing a summand $x^mg$ of $a+c$ by $x^m\ra(g)$ and keeping the rest of the monomials intact. By the nature of the three axioms, the number of monomials of the form $x^m g$ in $a+c$ is larger than in $b+c$ and each of the monomials in $x^m\ra(g)$ appears one time less in $a+c$ than in $b+c.$ Since these terms appear equal number of times in $c,$ this means that $a$ contains a monomial $x^mg$ and that $x^m\ra(g)$ is a summand of $b.$ Hence, $a=a'+x^mg$ and $b=a'+x^m\ra(g)$ for some $a'\in F_E^\Gamma$ so that $a\to_1 b.$ The case $b+c\to_1 a+c$ is similar and the induction step is analogous.
\end{proof}

\begin{remark}
Proposition \ref{cancellative} highlights an important difference between $M_E$ and $M_E^\Gamma$: while $M_E$ can be much larger than the positive cone of $G_E,$ the monoid $M_E^\Gamma$ is {\em equal to} the positive cone of $G_E^\Gamma.$ Thus,
the monoid $M_E$ can carry some information which is lost under formation of its Grothendieck group but $M_E^\Gamma$ carries {\em no additional information} than $G_E^\Gamma.$ In other words, using the language of \cite{Talented_monoid}, 
the group $G_E^\Gamma$ is equally ``talented'' as the monoid $M_E^\Gamma.$ 
\end{remark}

\subsection{The order}\label{subsection_order}
The relation $\sim$ on the monoid $F_E^\Gamma$ enables one to define a relation $\precsim$ as follows. 
\[a\precsim b\mbox{ if there is $c\in F^\Gamma_E$ such that }a+c\sim b\]
for all $a,b\in F^\Gamma_E.$ 
If $a\precsim b$ and $a\nsim b$, we write $a\prec b.$ Using Proposition \ref{cancellative}, it is direct to show that $a\prec b$ is equivalent with $a+c\sim b$ for some nonzero $c$ in $F_E^\Gamma.$ 

The relation $\precsim$ defines an order on $M_E^\Gamma$ given by   
\[[a]\leq [b]\;\;\mbox{ if and only if }\;\;a\precsim b.\]
It is direct to show that $\leq$ is reflexive and transitive. The antisymmetry holds by Proposition \ref{cancellative}. The relation $\leq$
induces an order on the Grothendieck group $G^\Gamma_E.$   

In \cite[Lemma 4.1]{Talented_monoid}, it is shown that $a\prec x^n a$ is not possible for any $a$ and any positive $n$ if $E$ is row-finite. After the lemma below, we show that this statement holds for an arbitrary graph in Proposition \ref{generalization_of_lemma_4_1}.  

\begin{lemma}
If $a\in F_E^\Gamma-\{0\}$ is such that $a\precsim x^na$ for some positive integer $n,$ then the following hold.
\begin{enumerate}[\upshape(1)]
\item No vertex in the support of $a$ is a sink. 
 
\item No vertex in the support of $a$ is an improper vertex. 

\item All vertices in the support of $a$ are regular (so $a$ is regular).
\end{enumerate}
\label{lemma_three_claims}
\end{lemma}
\begin{proof}
Since $a\precsim x^na,$ $a+b\sim x^na$ for some $b\in F_E^\Gamma.$ Then $a+b+x^nb\sim x^na+x^nb\sim x^{2n}a$ so, by induction, $a+\sum_{i=0}^kx^kb\sim x^{(k+1)n}a.$ Hence, we can find $n$ large enough so that $n$ is larger than the degrees of all monomials in a normal representation of $a.$
Assume that $n$ is such and that  $a+b\sim x^na$ for some $b\in F_E^\Gamma.$ By the Confluence Lemma \ref{confluence}(2), there is $c\in F^\Gamma_E$ such that $a+b \to c$ and $x^n a\to c.$ 

(1) Assume that a sink $v$ is in  $\supp(a)$ and let $\sum_{i=1}^l x^{m_i}v$ be the sum of all monomials in a normal representation of $a$ which contain $v.$ By construction, $m_i<n$ for every $i=1,\ldots, l.$ Since the relation $\to_1$ cannot be applied to $v,$ the relation $a+b\to c$ implies that $x^{m_i}v$ is a summand of $c$ for every $i=1,\ldots, l.$ On the other hand, the relation $x^n a\to c$ implies that every monomial of $c$ has degree larger than or equal to $n$ so $x^{m_1}v$ cannot be a summand of $c.$ This is a contradiction. 

(2) Assume that an improper vertex $q^v_Z$ is in $\supp(a)$ for some $v$ and some $Z.$ Let $\sum_{i=1}^l x^{m_i}q_{Z_i}$ be the sum of all monomials in a normal representation of $a$ which contain $q^v_{Z_i}$ for some nonempty and finite $Z_i\supseteq Z.$ Since an application of $\to_1$ does not change the power of a monomial with $q^v_W$ for some $W\supseteq Z,$ the relation $a+b\to c$ implies that $c$ contains a summand of the form $\sum_{i=1}^l x^{m_i}q_{W_i}$ for some $W_i\supseteq Z_i, i=1,\ldots, l.$ On the other hand, the relation $x^na\to c$ implies that every monomial of $c$ has degree larger than or equal to $n$ so $x^{m_1}q_{W_1}$ cannot be a summand of $c.$ This is a contradiction. 

(3) By part (1), to show that a vertex $v$ in the support of $a$ is regular, it is sufficient to show that $v$ is not an infinite emitter. Assume that an infinite emitter $v$ is in the support of $a$ and let $\sum_{i=1}^l x^{m_i}v$ be the sum of all monomials in a normal representation of $a$ which contain $v.$ Since axioms (A1) and (A3) are not applicable to any monomials with $v$ in them, the relation $a+b\to c$ implies that $\sum_{i=1}^l x^{m_i}g_i,$ where each $g_i$ is either $v$ or $q^v_Z$ for some $Z,$ is a summand in a normal representation of $c.$ On the other hand, the relation $x^n a\to c$ implies that  every monomial of $c$ has degree larger than or equal to $n$ so $x^{m_1}g_1$ cannot be a summand of $c$ which is a contradiction. 
\end{proof}

\begin{proposition}
The relation $a\prec x^na$ is not possible for any nonnegative $n$ and any $a\in F_E^\Gamma.$
\label{generalization_of_lemma_4_1}
\end{proposition}
\begin{proof}
Since $0\prec 0$ is false, it is sufficient to consider $a\neq 0.$ Also, since $a\prec a$ is false, it is sufficient to consider positive $n.$
Assume that $a\prec x^na$ for some positive $n$ and some nonzero $a\in F_E^\Gamma.$ By Lemma \ref{lemma_three_claims}, all elements in the support of $a$ are regular and proper vertices. Let $m$ be the maximum of degrees of the monomials in a normal representation of $a.$ If a monomial $x^lv$ in a normal representation of $a$ is such that $l<m,$ apply (A1) to $x^lv$ to replace this monomial by $\sum_{e\in \so^{-1}(v)}x^{l+1}\ra(e).$ We obtain an element $a_1$ such that $a_1\sim a$ so the relation $a_1\prec x^na_1$ also holds and, as a consequence, all vertices in the support of $a_1$ are regular also. Keep repeating this process until all monomials of some $a_k$ have the same degree $m$ so that we can write $a_k=x^m b$ where $b$ is a sum of regular vertices. Since $x^mb\prec x^{n+m}b$ we have that $b\prec x^nb$ and so $b+c\sim x^nb$ for some nonzero $c\in F_E^\Gamma.$  By the Confluence Lemma \ref{confluence}(2), there is $d$ such that $b+c\to d$ and $x^nb\to d.$  
The relation $x^nb\to d$ implies that $x^{-n}d\prec d$ so  $d\prec x^{n}d$ and all vertices in the support of $d$ are regular by Lemma \ref{lemma_three_claims}. Using the same argument as when obtaining $x^mb$ from $a,$ we can show that there is an element $f$ such that $d\to f$ and such that $f$ is a sum of monomials of the same degree $m'.$ Hence, $b+c\to d\to f$ and $x^nb\to d\to f.$  
Since $\to$ either increases the degree of a monomial or leaves it the same, the relation $x^nb\to f$ implies that $m'\geq n>0.$ 

Let $h=x^{-n}f$ so that $h$ is a sum of monomials of the same nonnegative degree $m'-n$ and that $b+c\to x^nh$ and $b\to h.$ We use induction on the length of a sequence for $b\to h$ to show that $h+c\to x^nh.$

If $b=h,$ the claim holds. Assume that it holds for length smaller than $k$ and let $b=b_0\to_1 b_1\to_1\ldots\to_1 b_k=h.$ Since $b$ is regular, $b\to_1 b_1$ is an application of (A1). Hence, $b=b'+v$ and $b_1=b'+\sum_{e\in \so^{-1}(v)}x\ra(e)$ for some regular vertex $v.$ 
Since the degree of every monomial in $x^nh=f$ is strictly larger than zero, $v$ has to be changed in the process of obtaining $x^nh$ from $b+c=b'+v+c.$ Reorder the terms of the sequence for $b+c\to x^nh$ so that an application of (A1) to $v$ is the first step. Hence,  
\[b+c=b'+v+c\to b'+\sum_{e\in \so^{-1}(v)}x\ra(e)+c=b_1+c\to x^nh.\]
We can now apply the induction hypothesis to $b_1$ to obtain that $h+c\to x^nh.$

Lastly, we show that the relation $h+c\to x^nh$ leads to a contradiction. Indeed, since $h$ is a sum of monomials of the same nonnegative degree and $n$ is strictly larger than zero, we have that $h+c\neq x^nh$ so at least one of the three axioms is used. If normal representations of $h$ and $c$ have $n_h$ and $n_c$ monomials respectively, then the number of terms in the resulting $x^nh$ is larger than or equal to 
$n_h+n_c.$ But since $x^nh$ has the same number of monomials as $h,$ we necessarily have that $n_c=0$ which implies that $c=0.$ This is a contradiction since $c$ is chosen to be nonzero such that $b+c\sim x^nb.$ 
\end{proof}

\subsection{Comparable, periodic, aperiodic and incomparable elements}\label{subsection_comparable}

Proposition \ref{generalization_of_lemma_4_1} implies that there are just two possibilities for $a\in F_E^\Gamma:$ either $a\succsim x^na$ for some positive $n$ or $a$ and $x^na$ are not comparable for any positive $n.$ In the case when  $a\succsim x^na$ for some positive $n$ we have that either $a\sim x^na$ or $a\succ x^na.$ We introduce the following terminology.   

\begin{definition} Let $a\in F_E^\Gamma.$
\begin{enumerate}
\item If $a\succsim x^na$ for some positive integer $n,$ the element $a$ is {\em comparable}.

\begin{enumerate}
\item[(1i)] If $a\sim x^na$ for some positive integer $n,$ the element $a$ is {\em periodic}.
 
\item[(1ii)] If $a\succ x^na$ for some positive integer $n,$ the element $a$ is {\em aperiodic}.
\end{enumerate}

\item If $a$ and $x^na$ are not comparable for any positive integer $n,$ the element $a$ is {\em incomparable}.
\end{enumerate}

For $[a]\in M_E^\Gamma,$ we say that $[a]$ is comparable, periodic, aperiodic or incomparable if any $b$ such that $a\sim b$ is such.
\label{definition_of_comparable}
\end{definition}

Note that 0 is periodic by this definition. An element of $F_E^\Gamma$ clearly cannot be both comparable and incomparable. We also note that a comparable element of $F_E^\Gamma$ cannot be both periodic and aperiodic. Indeed, if $x^m a\sim a\succ x^n a$ for some positive integers $m$ and $n,$ let $n$ be the least positive integer such that $a\succ x^na$. Since $x^m a\succ x^n a$ implies $x^{m-n}a\succ a,$ $m-n$ is negative by Proposition \ref{generalization_of_lemma_4_1} so $n>m$. On the other hand, the relation $x^m a\sim a\succ x^na$ also implies that $a\sim x^{-m}a\succ x^{n-m}a$ so $n-m\geq n$ by the assumption that $n$ is the smallest possible such that $a\succ x^na.$ The relation $n-m\geq n$ implies that $m\leq 0$ which is in contradiction with the assumption that $m$ is positive.

\subsection{Stationary elements}\label{subsection_stationary}
Next, we prove a series of claims which bring us to Theorem \ref{comparable}. Lemma \ref{comparable_connects_to_stationary} leads us to the notion of a stationary element introduced in Definition \ref{definition_stationary}.

\begin{lemma} Let $a\in F_E^\Gamma-\{0\}$  be such that $a\sim x^na+b$  for some positive integer $n$ and some $b\in F_E^\Gamma.$ There are $c\in F_E^\Gamma-\{0\}$ and $d\in F_E^\Gamma$ such that $c\to x^nc+d,$ $a\to c$ and $b\to d.$
\label{comparable_connects_to_stationary}
\end{lemma}
Note that the assumption of the lemma is exactly that $a$ is comparable, the case $b=0$ corresponds exactly to the case that $a$ is periodic, and the case $b\neq 0$ to the case that $a$ is aperiodic.
\begin{proof}
Since $a\sim x^na+b\sim x^{2n}a+x^nb+b\sim \ldots,$ we can choose $n$ as large as needed. Let us choose $n$ larger than the degree of every monomial in a normal representation of $a.$ 

By the Confluence Lemma \ref{confluence}(2), $a\to f$ and $x^na+b\to f$ and by the Refinement Lemma \ref{confluence}(1), $f=f_1+f_2$ such that $x^na\to f_1$ and $b\to f_2.$ Let $c=x^{-n}f_1$ so that $a\to x^{-n}f_1=c$ and that $a\to f= x^nc+f_2.$ 

We use induction on $k$ for $a\to^k c.$ If $k=0,$ then $a=c.$ Let $d=f_2$ so that $b\to d.$ Assuming the induction hypothesis, let us consider $a\to^k c$ with $a=a_0\to_1 a_1\to_1\ldots \to_1 a_k=c.$ Let $a=a'+x^mg$ for some generator $g$ such that $a_1=a'+x^m\ra(g).$ Consider the following two cases for the relation $a\to x^nc+f_2.$
\begin{enumerate}[1.]
\item There is an application of the same axiom used for $a\to_1 a_1$ to $x^mg$ at some point such that $x^mg$ is not changed prior to this point. Changing the order of applications of axioms in the sequence for $a\to x^nc+f_2,$ we can assume that this application of the axiom happened first. In this case $a\to a_1\to x^nc+f_2.$ Thus, we can apply the induction hypothesis to $a_1$ instead of $a$ and obtain the relation $c\to x^nc+d$ for some $d$ such that $f_2\to d.$ Hence, $b\to f_2\to d.$ 

\item There is no application of the axiom used for $a\to_1 a_1$ to $x^mg$ at any point. Since $n$ is larger than $m,$ then $x^mg$ has to be a summand of $f_2.$ Say $f_2=d'+x^mg.$ Then $a=a'+x^mg\to x^nc+d'+x^mg.$ 
Replacing the terms $x^mg$ by $x^{m}\ra(g)$ on both sides of the relation $\to,$ we obtain that 
$a_1=a'+x^{m}\ra(g)\to x^nc+d'+x^{m}\ra(g).$
Since we have $a_1\to^{k-1} c,$ we can apply the induction hypothesis to $a_1$ and obtain that $c\to x^nc+d$ for some $d$ such that $d'+x^{m}\ra(g)\to d.$ Hence, $b\to f_2=d'+x^mg\to d'+x^{m}\ra(g)\to d.$
\end{enumerate}
\end{proof}

The properties of an element such as element $c$ of Lemma \ref{comparable_connects_to_stationary} are significant in the characterization of a comparable element so we assign a name to such an element. 

\begin{definition}
An element $a\in F_E^\Gamma-\{0\}$ is a {\em stationary} element if $a\to x^na+b$ for some positive integer $n$ and some $b\in F_E^\Gamma.$ 
\label{definition_stationary} 
\end{definition}

\begin{example}
\begin{enumerate}
\item 
If $E$ is the graph  $\xymatrix{{\bullet}^v \ar[r] & {\bullet}^w\ar@(ru,rd) }\;\;\;\;,$ then $w$ is stationary since $w\to xw$. One can directly check that if $v\to a$ for some $a\in F_E^\Gamma,$ then either $a=v$ or $a$ is of the form $x^nw$ for some positive integer $n.$ Hence, $v$ is not stationary.   
 
\item Let $E$ be the Toeplitz graph $\;\;\;\;\xymatrix{{\bullet}^v\ar@(lu,ld)  \ar[r] & {\bullet}^w}$ 
and $a=v+w\in F_E^\Gamma.$ Since $a=v+w\to xv+xw+w=x(v+w)+w=xa+w,$ $a$ is stationary. Note that $b=v+xw$ has the same support as $a$ but $b$ is not stationary. Indeed, if $b\to c,$ then $c=b$ or $c=x^nv+x^nw+x^{n-1}w+\ldots+xw+w+xw$ for some positive $n.$ So, assuming that $b\to x^nb+d$ for some $d$ and positive $n$ leads to a contradiction. 

We note also that adding $xv$ to $b,$ we obtain a stationary element again since it is a sum of stationary elements $x(v+w)$ and $v.$
\end{enumerate}
\label{example_stationary}
\end{example}

The next lemma describes the support of a stationary element. Recall that a generator $g$ is on a cycle if $g\cto^p g$ for some $p$ with $|p|>0.$ 

\begin{lemma} Let $a\in F_E^\Gamma$ be stationary such that $a\to x^na+b$ for some positive integer $n$ and some $b\in F_E^\Gamma$. 
\begin{enumerate}[\upshape(1)]
\item For any positive integer $k,$ \[a\to x^{kn}a+\sum_{i=0}^{k-1}x^{in}b.\]

\item The support of $a$ contains an element which is on a cycle.

\item Each element of the support of $a$ which is not on a cycle is on a path exiting a cycle which contains another element of $\supp(a).$ \footnote{This condition can be described also in terms of the {\em tree} $T(g)=\{h\mid g\cto h\}$ of a generator $g$ as follows: $\supp(a)\subseteq \bigcup \{T(g)\mid  g\in \supp(a)$ and $g$ on a cycle$\}.$}

\item Each element of the support of $a$ is either on a cycle or on a path exiting a cycle which contains another element of $\supp(a).$ 
\end{enumerate} 
\label{support_of_stationary} 
\end{lemma}
\begin{proof}
To show (1), note that if $a\to x^na+b,$ then
\[a\to x^na+b \to  x^{2n}a+x^nb+b\to x^{3n}a+x^{2n}b+x^nb+b\to \ldots \to x^{kn}a+\sum_{i=0}^{k-1} x^{in}b.\]

To show (2), we use part (1) to choose $n$ larger than $k-m$ for any degrees $k$ and $m$ of any monomials in a normal representation of $a.$ Let $l$ be the number of monomials in a normal representation of $a.$ 

If all generators in $\supp(a)$ are on cycles, there is nothing to prove. So, suppose that there is $g_1\in \supp(a)$ such that $x^{m_1}g_1$ is a monomial of $a$ and that $g_1$ is not on a cycle. 
Let $a=a_1+x^{m_1}g_1.$ By the Refinement Lemma \ref{confluence}(1), there are $c_{11}, c_{12}$ such that
\begin{center}
$a_1+x^{m_1}g_1\to x^na_1+x^{n+m_1}g_1+b=c_{11}+c_{12},\;\;$ $a_1\to c_{11}$ and $x^{m_1}g_1\to c_{12}.$ 
\end{center} 
The monomial $x^{n+m_1}g_1$ is a summand of either $c_{11}$ or $c_{12}.$ In the second case, $x^{m_1}g_1\to x^{n+m_1}g_1+c$ for some $c$ and Corollary \ref{converse_connecting} implies that there is a path of length $n>0$ from $g_1$ to $g_1$ which means that $g_1$ is on a cycle. This is a contradiction with the choice of $g_1.$ Hence, $x^{n+m_1}g_1$ is a summand of $c_{11}.$ This implies that $c_{11}\neq 0$ and so $a_1\neq 0$ also which means that $l>1$ and $a_1$ has $l-1$ terms. 

By Corollary \ref{decomposing_the_initial_term_lemma}, there is a monomial $x^{m_2}g_2$ of $a_1$ such that $a_1=a_2+x^{m_2}g_2$ (so $a_2$ has $l-2\geq 0$ terms) and that $x^{m_2}g_2\to x^{n+m_1}g_1+c$ for some $c.$ The choice of $n$ guarantees that $n+m_1-m_2>0$ so that there is a path of positive length from $g_2$ to $g_1$ by Corollary \ref{converse_connecting}. 
If $g_2$ is on a cycle, we are done. If not, consider whether the term $x^{n+m_2}g_2$ is a summand of $c_{11}$ or $c_{12}.$ If it is a summand of $c_{12},$ then $x^{m_1}g_1\to x^{n+m_2}g_2+d$ for some $d$ and so there is a path of positive length from $g_1$ to $g_2.$ As there is a path of positive length from $g_2$ to $g_1,$ $g_1$ is on a cycle. Since this is not the case,  $x^{n+m_2}g_2$ is a summand of $c_{11}.$ 

Apply the Refinement Lemma \ref{confluence}(1) again to decompose $c_{11}$ as $c_{21}+c_{22}$ such that $a_2\to c_{21}$ and $x^{m_2}g_2\to c_{22}.$ Since $g_2$ is not on a cycle, $x^{n+m_2}g_2$ is a summand of $c_{21}$ which implies that $c_{21}\neq 0$ and so $a_2\neq 0$ which means that $l-2>0.$  
By Corollary \ref{decomposing_the_initial_term_lemma}, there is a summand $x^{m_3}g_3$ of $a_2$ such that $a_2=a_3+x^{m_3}g_3$ (so that $a_3$ has $l-3\geq 0$ terms) and that $x^{m_3}g_3\to x^{n+m_2}g_2+d$ for some $d.$ The choice of $n$ guarantees that $n+m_2-m_3$ is positive so we can conclude that there is a path of positive length from $g_3$ to $g_2$ by Corollary \ref{converse_connecting}. 

If $g_3$ is on a cycle, we are done. If not, the term $x^{n+m_3}g_3$ must be a summand of $c_{21}$ as otherwise $g_3$ is on a cycle which is not the case. So, $x^{n+m_3}g_3$ is a summand of $c_{21},$ $a_3+x^{m_3}g_3\to c_{21},$ and we can continue the decomposition process $c_{21}=c_{31}+c_{32}$ as in the previous step. 

Since $l$ is finite, this process eventually stops. If it stops at the $k$-th step, $g_k$ is on a cycle and (2) holds. 

Note that the proof of part (2) implies that if $g_1$ is not on a cycle, then $g_1$ is on a path leaving a cycle which contains $g_k.$ This is because the proof shows that there is a path from $g_{i+1}$ to $g_i$ for all $i=1,\ldots, k-1.$ Hence, this automatically shows part (3). Part (4) is a direct corollary of part (3).
\end{proof}

The last part of Lemma \ref{support_of_stationary} describes the support of a stationary element. The properties of such set are relevant and we introduce some terminology for it. First, we say that a finite and nonempty set of generators of $F_E^\Gamma$ is {\em stationary} if every $g\in V$ is either on a cycle or on a path exiting a cycle which contains some generator $h\in V.$ By Lemma \ref{support_of_stationary}, the support of every stationary element is a stationary set. 

For a stationary set $V,$ let $V_c$ denote the set of those $g\in V$ which are on cycles (thus $V_c\neq \emptyset$). We say that $V_c$ is the {\em core} of $V$ and that $g\in V_c$ is a {\em core  generator}. We say that the cycles which contain core generators are the {\em core cycles} of $V$. Let $V_e$ denote $V-V_c$ (so $V_e$ is possibly empty). We call this set the {\em exit set} of $V$ and we say that $g\in V_e$ is an {\em exit generator.}

For a core generator $g\in V_c$, let $n_g$ be the minimum of the set of lengths of cycles on which $g$ is. Let $n$ be the least common multiple of $n_g$ for $g\in V_c.$ We show that $n$ has a special significance for a stationary set $V$ which consists of core generators only so we call it {\em the core period} of such $V.$  

If $a$ is stationary, let $a=a_c+a_e$ such that the support of $a_c$ is $\supp(a)_c$ and the support of $a_e$ is $\supp(a)_e$ (thus $a_c\neq 0$ and $a_e$ is possibly zero). We call $a_c$ and $a_e$ {\em the core part} and {\em the exit part} of $a$ respectively. 

The next example illustrates these newly introduced concepts. 

\begin{example}
\begin{enumerate}
\item If $E$ is the graph  $\xymatrix{{\bullet}^v \ar[r] & {\bullet}^w\ar@(ru,rd) }\;\;\;\;,$ then the stationary element $w$ has the core part $w$ and the exit part 0. The set $\{w\}$ is a stationary set with the core equal to the entire set $\{w\}.$ The core period of the core $\{w\}$ is 1. 
   
\item If $E$ is the Toeplitz graph $\;\;\;\;\xymatrix{{\bullet}^v\ar@(lu,ld)  \ar[r] & {\bullet}^w},$ then the stationary element $a=v+w$ has the core part $v$ and the exit part $w.$ The stationary set $\{v,w\}$ has the core $\{v\}$ and the exit set $\{w\}.$ The core period of the core $\{v\}$ is 1. 
\end{enumerate}
\label{example_core}
\end{example}

If $a\in F_E^\Gamma$ is such that each $g\in \supp(a)$ is on a cycle, then $\supp(a)$ is a stationary set by definition and $a=a_c.$ In the next lemma, we show that such element $a$ is necessarily stationary.  

\begin{lemma}
Let $a\in F_E^\Gamma-\{0\}$ be such that each $g\in \supp(a)$ is on a cycle, and let $n$ be the core period of $\supp(a).$ The following hold. 
\begin{enumerate}[\upshape(1)]
\item The element $a$ is  stationary and $a\to x^na+b$ for some $b\in F_E^\Gamma$. 

\item The element $a$ is periodic if and only if the core cycles have no exits.  
\end{enumerate}
\label{stationary_support_cyclic}
\end{lemma}
\begin{proof}
If $g\in \supp(a),$ then $g\cto^{c_g} g$ where $c_g$ is a cycle of length $n_g,$ where $n_g$ is the minimum of the set of lengths of cycles which contain $g.$ Hence, 
$g\to x^{n_g}g+b_g'$ for some $b_g'\in F_E^\Gamma$ such that  $b_g'=0$ if and only if $c_g$ has no exits. Since $n$ is a multiple of $n_g,$ $g\to x^ng+b_g$ for some $b_g$ such that $b_g=0$ if and only if $b_g'=0.$

If $a=\sum_{j=1}^l x^{k_j}g_j$ is a normal representation of $a,$ then we have that $x^{k_j}g_j\to x^{n+k_j}g_j+x^{k_j}b_{g_j}.$ Adding these relations together produces 
\[a\to \sum_{j=1}^l x^{n+k_j}g_j+\sum_{j=1}^l x^{k_j} b_{g_j}=x^n\sum_{j=1}^l x^{k_j}g_j+\sum_{j=1}^l x^{k_j}b_{g_j}=x^na+b\]
for $b=\sum_{j=1}^l x^{k_j}b_{g_j}$ so (1) holds. To show (2), note that $a$ is periodic if and only if $b=0$ and $b=0$ if and only if any core cycle has no exits. 
\end{proof}

We note the following corollary of Lemmas \ref{comparable_connects_to_stationary}, \ref{support_of_stationary}, and \ref{stationary_support_cyclic}. 

\begin{corollary}
The following conditions are equivalent.  
\begin{enumerate}[\upshape(1)]
\item There is a comparable generator of $F_E^\Gamma.$ 
\item There is a nonzero comparable element of $F_E^\Gamma.$  
\item The graph $E$ has a cycle. 
\end{enumerate}
\label{exists_comparable}
\end{corollary}
\begin{proof} 
The implication (1) $\Rightarrow$ (2) is direct. If (2) holds, there is a stationary element $a$ by Lemma \ref{comparable_connects_to_stationary}. Since $a_c\neq 0$ by Lemma \ref{support_of_stationary}, there is at least one core cycle so (3) holds. If (3) holds, any vertex of a cycle is a comparable generator of $F_E^\Gamma$ by Lemma \ref{stationary_support_cyclic} so (1) holds.   
\end{proof}

\subsection{The Core Lemma}\label{subsection_core_lemma}

The following lemma highlights an important property of a stationary element and justifies our terminology ``core'' -- if $a$ is stationary and $x^na$ can be produced from $a$ with some possible ``change'' $b$, then $x^{kn}a$, for some positive $k,$ can be produced by using the core part $a_c$ only with possibly some other ``change'' $c$ such that $c=0$ and $a_e=0$  if and only if $b=0.$

\begin{lemma}(The Core Lemma)
Let $a\in F_E^\Gamma$ be a stationary element with the core part $a_c$ and the exit part $a_e.$  
If $a\to x^na+b$ for some positive integer $n$ and some $b\in F_E^\Gamma,$ then $a_c\to x^{kn}a+c$ for some positive integer $k$ and some $c\in F_E^\Gamma$ such that $c+a_e\sim \sum_{i=0}^{k-1}x^{in}b.$
\label{core_lemma}
\end{lemma}
\begin{proof} 
If $a_e=0,$ the claim trivially holds with $k=1$ and $c=b.$ If $a_e\neq 0,$ let $V=\supp(a)$ so that $V_e$ is nonempty. Let also  $V_c=V_c'\cup V_c''$ where $V_c'$ consists of the core generators in $V$ such that no exit generator connects to them and $V_c''$ consists of the core generators in $V$ such that some exit generators connect to them. By these definitions, no $g\in V_c''$ connects to any $h\in V_c'$ (otherwise $h$ would be in $V_c''$). Also, note that $V_c'$ is nonempty since otherwise some exit generator would be on a cycle which would make it a core, not an exit generator. Let also $a_c=a_c'+a_c''$ so that $\supp(a_c')=V_c'$ and $\supp(a_c'')=V_c''.$ Choose $n$ to be larger than the difference of degrees of any two monomials in a normal representation of $a$ by using Lemma \ref{support_of_stationary}(1) if $n$ is not already such. 

We construct a sequence of finite acyclic graphs $F_0\supsetneq F_1\supsetneq\ldots\supsetneq F_l\supsetneq \emptyset$ such that the sequence  terminates exactly when the claim is shown. 

{\bf Graph $\mathbf{F_0.}$}
Let us define a graph $F_0$ such that $V_e$ is the set of vertices of $F_0$ and that there is an edge from $g$ to $h$ for some $g,h\in V_e$ if $g$ connects to $h$ in $E.$ Since no $g\in V_e$ is on a cycle, the graph $F_0$ is acyclic.
Since $F_0$ is a finite and acyclic graph, it has a source by Lemma \ref{sources}. Let $V_{e0}$ be the set of sources of $F_0$ and $a_e=a_{e0}+a_{e0}'$ such that $\supp(a_{e0})=V_{e0}$ and $\supp(a_{e0}')=V_e-V_{e0}.$

By the Refinement Lemma \ref{confluence}(1), there are  $a_1,a_2,a_3\in F_E^\Gamma$ such that
\[a=a_c'+(a_c''+a_{e0}')+a_{e0}\to x^na+b=a_1+a_2+a_3\;\mbox{ and }\;a_c'\to a_1,\; a_c''+a_{e0}'\to a_2,\; a_{e0}\to a_3.\]
If $x^mg$ is any monomial of $x^na_{e0}$ for $g\in V_{e0},$  then $x^mg$ is a summand of either $a_1, a_2$ or $a_3.$ By Corollary \ref{decomposing_the_initial_term_lemma} and by the choice of $n,$ if $x^mg$ is a summand of $a_3$ then either $g$ is on a cycle or there is a path from another source of $F_0$ to $g$ and each of these options leads to a contradiction. If $x^mg$ is a summand of $a_2,$ then there is either a nontrivial path from some $g'\in V_e$ to $g$ or a nontrivial path from some $h\in V_c''$ to $g$ also by Corollary \ref{decomposing_the_initial_term_lemma} and by the choice of $n$. In the second case, there is $g'\in V_e$ and a path from $g'$ to $h$ and, hence, a nontrivial path from $g'$ to $g$ as well. Thus, both cases lead to a contradiction since $g$ is a source of $F_0.$ Hence, $x^mg$ has to be a summand of $a_1.$ Since the monomial $x^mg$ was arbitrary, $x^na_{e0}$ is a summand of $a_1.$ 
In addition, if $x^mh$ is any monomial of $x^na_c',$  $x^mh$ is a summand of $a_1$ also. Indeed, assuming that $x^mh$ is a summand of either $a_3$ or $a_2$ implies that $h$ is in $V_c''$ not $V_c'.$ Hence, for some $b_0\in F_E^\Gamma,$
\[a_c'\to a_1=x^na_c'+ x^na_{e0}+b_0.\]

If $a_{e0}'=0,$ we claim that the process is complete. In this case, $a_e=a_{e0}.$ The support of $a_c''$ consists of core generators so $a_c''$ is stationary by Lemma \ref{stationary_support_cyclic}. Let $m$ be the least common multiple of $n$ and the core period of $a_c''$ and let $m=kn.$
Let $b''_0$ be such that $a_c''\to x^{kn} a_c''+b''_0.$ After repeated use of the relation $a_c'\to x^na_c'+x^na_e+b_0$ for $k$ times, 
we have that  
$$a_c'\to x^{kn}a_c'+x^{kn}a_e+\sum_{i=1}^{k-1}x^{in}a_e+\sum_{i=0}^{k-1} x^{in}b_0.$$ Thus, 
\[
a_c=a_c'+a_c'' \to  x^{kn}a_c'+x^{kn}a_e+\sum_{i=1}^{k-1}x^{in}a_e+\sum_{i=0}^{k-1} x^{in}b_0+ x^{kn} a_c''+b''_0= 
x^{kn}a+\sum_{i=1}^{k-1}x^{in}a_e+\sum_{i=0}^{k-1} x^{in}b_0+b''_0=x^{kn}a+c
\] for $c=\sum_{i=1}^{k-1}x^{in}a_e+\sum_{i=0}^{k-1} x^{in}b_0+b''_0.$ Thus, $a=a_c+a_e\to x^{kn}a+c+a_e.$ On the other hand, 
$a\to x^{kn}a+\sum_{i=0}^{k-1}x^{in}b$ holds by part (1) of Lemma \ref{support_of_stationary}. Thus,  
$$x^{kn}a+c+a_e\sim x^{kn}a+\sum_{i=0}^{k-1}x^{in}b\;\;\;\mbox{ which implies }\;\;\;c+a_e\sim\sum_{i=0}^{k-1}x^{in}b.$$ 

If $a_{e0}'\neq 0,$ we construct $F_1.$ 

{\bf Graph $\mathbf{F_1.}$} Let $F_1$ be the graph obtained by eliminating the sources and all edges they emit from $F_0.$ Then $F_1$ is a finite acyclic graph which is a proper subgraph of $F_0.$ Let $V_{e1}$ be the set of the sources of $F_1$ and $a_e=a_{e0}+a_{e1}+a_{e1}'$ be such that $\supp(a_{e1})=V_{e1}$ and $\supp(a_{e1}')=V_e-V_{e0}-V_{e1}.$ Let also $a_c''=a_{c0}''+a_{c1}''$
such that $a_{c0}''$ consists of those monomials $x^mh$ of $a_c''$ such that $g\cto h$ for some $g\in V_{e0}$ and $a_{c1}''$ consists of all other monomials of $a_c''.$
Using the Refinement Lemma \ref{confluence}(1) again, there are $a_1', a_2', a_3'\in F_E^\Gamma$ such that 
\[a=(a_c'+a_{c0}''+a_{e0})+(a_{c1}''+a_{e1}')+a_{e1}\to x^na+b=a_1'+a_2'+a_3'\] 
and that $a_c'+a_{c0}''+a_{e0}\to a_1'$, $a_{c1}''+a_{e1}'\to a_2'$, $a_{e1}\to a_3'.$ 
If $x^mg$ is any summand of $x^na_c'+x^na_{c0}''+x^na_{e0}+x^na_{e1},$ we can repeat the arguments from before to show that the assumption that $x^mg$ is a summand of $a_2'$ or $a_3'$ leads to a contradiction. Hence, $x^mg$ is a summand of $a_1'$ and so  
\[a_c'+a_{c0}''+a_{e0}\to a_1'=x^na_c'+x^na_{c0}''+x^na_{e0}+x^na_{e1}+b_1'\]
for some $b_1'\in F_E^\Gamma.$ If $k_1n$ is the least common denominator of $n$ and the core period of $a_{c0}''$, there is $b_1''\in F_E^\Gamma$ such that $a_{c0}''\to x^{k_1n}a_{c0}''+b_1''.$ Using the last two relations and the relation $a_c'\to x^na_c'+ x^na_{e0}+b_0$ from the first step for $k_1$ times, we have that 
\[a_c'+a_{c0}''\to x^{k_1n}(a_c'+ a_{e0})+\sum_{i=1}^{k_1-1}x^{in}a_{e0}+\sum_{i=0}^{k_1-1}x^{ni}b_0+x^{k_1n}a_{c0}''+b_1''\to\]
\[
x^{(k_1+1)n}(a_c'+a_{c0}''+a_{e0}+a_{e1})+x^{k_1n}b_1'+\sum_{i=1}^{k_1-1}x^{in}a_{e0}+\sum_{i=0}^{k_1-1}x^{ni}b_0+b_1''
=x^{(k_1+1)n}(a_c'+a_{c0}''+a_{e0}+a_{e1})+b_1\]
for $b_1=x^{k_1n}b_1'+\sum_{i=1}^{k_1-1}x^{in}a_{e0}+\sum_{i=0}^{k_1-1}x^{ni}b_0+b_1''.$

If $a_{e1}'=0,$ then $a_e=a_{e0}+a_{e1}.$ Let $kn$ be the least common multiple of $(k_1+1)n$ and the core period of $a_{c1}''.$ 
Arguing as in the case $a_{e0}'=0,$ we have that 
$a_c\to x^{kn}a+c$ for some $c\in F_E^\Gamma$ such that $\sum_{i=0}^{k-1}x^{in}b\sim c+a_e$ and this finishes the proof. If $a_{e1}'\neq 0,$  we construct $F_2$ and continue the process.  

This process eventually terminates since $V_e$ is a finite set. Hence, there is a positive integer $l$ such that $a_{el}'=0$ so that 
$a_c\to x^{kn}a+c$ for some $k$ and some $c.$ The relations $a\to x^{kn}a+\sum_{i=0}^{k-1}x^{in}b$ and $a\to x^{kn}a+c+a_e$ imply that $\sum_{i=0}^{k-1}x^{in}b\sim c+a_e$ which proves the lemma. 
\end{proof}

The Core Lemma has the following corollary, characterizing a stationary and periodic element, which we use  in the proof of  Theorem \ref{periodic}.   

\begin{corollary}
A stationary element $a$ is periodic if and only if the support of $a$ consists of regular vertices on cycles without exits. 
\label{stationary_periodic} 
\end{corollary}
\begin{proof}
Let $a$ be such that $a\to x^na+b$ for some $b$ and positive $n$. If $a$ is periodic, then $b=0.$ 
By the Core Lemma \ref{core_lemma}, $a_c\to x^{kn}a+c$ for some $k$ and some $c$ such that $\sum_{i=0}^{k-1}x^{in}b\sim a_e+c.$ So $b=0$ implies that $a_e=0$ (and $c=0$). Hence, $a=a_c.$ This enables us to use Lemma \ref{stationary_support_cyclic} which implies that the support of $a$ consists of generators on cycles without exits so that these generators are regular vertices. 

For the converse, assume that the support of $a$ consists of core vertices on cycles without exits. If $n$ is the core period, then $a\to x^na$ so $a$ is both stationary and periodic. 
\end{proof}

\subsection{The stationary-partition}\label{subsection_partition}

By Lemma \ref{support_of_stationary}, the support of a stationary element is a stationary set. By Lemma \ref{stationary_support_cyclic}, the converse is true if a stationary set contains no exit generators. It would be convenient to have the converse of this fact in general. However, the exit generators can complicate the situation as the next example shows. 
 
Part (2) of Example \ref{example_stationary} shows that we need additional requirements for any element with a stationary support to be stationary. In particular, these requirements impose restrictions on powers of $x$ which appear in the normal form of such element. 

Let $a$ be stationary such that $a\to x^na+b$ holds for some positive $n$ and some $b.$  If $a=\sum_{i=1}^{k}x^{m_i}g_i$ is a normal representation of $a,$ by repeated use of the Refinement Lemma \ref{confluence}(1), there are mutually disjoint subsets $I_1, \ldots, I_k$ of $\{1, 2, \ldots, k\}$ whose union is $\{1, 2, \ldots, k\}$ and there are $b_1,\ldots, b_k$ such that  
\[a= \sum_{i=1}^{k}\sum_{j\in I_i} x^{m_j}g_j\;\mbox{ and }\;\; b=\sum_{i=1}^{k}b_i\]
and that for every $i=1, \ldots, k$ 
\begin{equation}\label{arrow_equation} \tag{Rel. 1}
x^{m_i}g_i\to \sum_{j\in I_i} x^{m_j+n}g_j+b_i.
\end{equation}
The set $I_i$ can be empty if $i$ is in $I_{i'}$ for some $i'\neq i$ (see also Example \ref{example_core_partition} below). If $I_i\neq\emptyset,$ Corollary \ref{converse_connecting} applied to relation (\ref{arrow_equation}) ensures the existence of a path $p_{ij}$ connecting $g_i$ and $g_j$ such that  
\begin{equation}\label{combinatorial_equation} \tag{Rel. 2}
m_i+|p_{ij}|=m_j+n.
\end{equation}
By Lemma \ref{support_of_stationary}(1), we can choose $n$ such  that $n>m_i-m_j$ so that $|p_{ij}|=n+m_j-m_i>0$ for all $i,j=1,\ldots, k.$ The requirement that $p_{ij}$ has positive length justifies the following definition and implies direction $\Rightarrow$ of Proposition \ref{key_lemma}. 

\begin{definition} \label{core_partition}
Let $a\in F_E^\Gamma$ have a stationary support $V$ and a normal representation $a=\sum_{i=1}^{k}x^{m_i}g_i.$ We say that $a$ {\em has a stationary-partition} if there is a positive integer $n,$ mutually disjoint subsets $I_1, \ldots, I_k$ of $\{1, 2, \ldots, k\}$ with $\bigcup_{i=1}^k I_i=\{1, 2, \ldots, k\}$ and paths $p_{ij}$ of {\em positive length} for $i=1,\ldots, k$ and $j\in I_i$ with $\so(p_{ij})=g_i$ and $\ra(p_{ij})=g_j$ and such that relation (\ref{combinatorial_equation}) holds for each $i=1, \ldots, k$ and $j\in I_i.$
\end{definition}

The following example shows that a stationary-partition does not have to be unique. 
\begin{example}
Let $E$ be the following graph $\;\;\xymatrix{{\bullet}^{v_1}\ar@(lu,ld)  \ar[r] & {\bullet}^{v_2}\ar@(ru,rd)}\;\;\,.$
Then $v_1+v_2$ is stationary since $v_1\to xv_1+xv_2$ and so $v_1+v_2\to x(v_1+v_2)+v_2$ and $k=2$ in this case.  
We can take $I_1=\{1,2\}$ and $I_2=\emptyset $ since $v_1$ ``produces'' both terms of $x(v_1+v_2).$ In this case, relations (\ref{arrow_equation}) are 
\[v_1\to xv_1+xv_2\;\;\mbox{ and }\;\;v_2\to v_2.\]
However, $v_2\to xv_2$ also, so the summand $xv_2$ can be ``produced'' by $v_2$ also. Hence, $v_1+v_2$ is stationary also because $v_1+v_2\to xv_1+xv_2+v_2\to xv_1+xv_2+xv_2=x(v_1+v_2)+xv_2.$ 
So, we can also take $I_1=\{1\},$ $I_2=\{2\}.$ In this case, relations (\ref{arrow_equation}) are 
\[v_1\to xv_1+xv_2\;\;\mbox{ and }\;\;v_2\to xv_2.\]
\label{example_core_partition}
\end{example}

We characterize a stationary element in terms of the properties of the generators in its support which is the final and key step towards Theorem \ref{comparable}. 

\begin{proposition}
Let $a\in F_E^\Gamma$ be an element such that $\supp(a)=V$ is stationary. Then $a$ is stationary if and only if $a$ has a stationary-partition. 
\label{key_lemma} 
\end{proposition}
\begin{proof}
We showed that direction $\Rightarrow$ holds before Definition \ref{core_partition}. To summarize, if $a=\sum_{i=1}^{k}x^{m_i}g_i\to x^na+b$ holds for some $n$ and some $b,$ use Lemma \ref{support_of_stationary}(1) to choose $n>m_i-m_j$ for all $i,j=1,\ldots, k.$ Repeated use of the Refinement Lemma \ref{confluence}(1) produces required sets $I_1, \ldots, I_k$ such that relations (\ref{arrow_equation}) hold for $i=1, \ldots, k.$ Using Corollary \ref{converse_connecting} produces paths $p_{ij}$ and our choice of $n$ ensures that the paths $p_{ij}$ have positive length so that relations (\ref{combinatorial_equation}) hold. Thus, $a$ has a stationary-partition. 

Conversely, let $a$ have a stationary-partition and let $n$, $I_1,\ldots, I_k$ and $p_{ij}$ be as in Definition \ref{core_partition}. Starting with $x^{m_i}g_i$ and applying the axioms following the paths $p_{ij}$ for all $i=1,\ldots, k$ and all $j\in I_i,$ we obtain $x^{m_i}g_i\to \sum_{j\in I_i} x^{m_i+|p_{ij}|}g_j+b_i$ for some $b_i\in F_E^\Gamma$ for $i=1,\ldots, k.$ By relations (\ref{combinatorial_equation}), 
\[x^{m_i}g_i\to \sum_{j\in I_i} x^{m_i+|p_{ij}|}g_j+b_i =\sum_{j\in I_i} x^{m_j+n}g_j+b_i\]
which shows that relations (\ref{arrow_equation}) hold for all $i.$ Adding these relations together produces 
\[a=\sum_{i=1}^{k}x^{m_i}g_i\to \sum_{i=1}^{k}\left(\sum_{j\in I_i} x^{m_j+n}g_j+b_i\right)=x^n\sum_{i=1}^{k}\sum_{j\in I_i} x^{m_j}g_j+\sum_{i=1}^{k}b_i=x^na+\sum_{i=1}^{k}b_i
\]
where the last equality holds since  $I_1,\ldots, I_k$ are disjoint and their union is $\{1,\ldots, k\}.$ Letting $b=\sum_{i=1}^{k}b_i,$ we have that $a\to x^na+b.$ Hence, $a$ is stationary.
\end{proof}

\begin{remark}
Since the relation $a\to x^na+b$ holds for some $n$ and $b$ if and only if $a_c\to x^ma+c$ holds for some $m$ and $c$ by the Core Lemma \ref{core_lemma}, we can also consider a partition of $x^ma+c$ based on a normal representation of $a_c$ instead of $a.$ If Definition \ref{core_partition} is modified accordingly, Proposition \ref{key_lemma} can be formulated to state that $a$ is stationary if and only if $a$ has a partition based on its core part $a_c$.  
\end{remark}

Proposition \ref{key_lemma} also reaffirms Lemma \ref{stationary_support_cyclic} since if an element $a$ has the stationary support consisting of core generators only, then $a$ has a stationary-partition. Indeed, if $a=a_c,$ one can take $n$ to be the core period and $I_i=\{i\}.$ If $n_{g_i}$ is the minimum of the set of lengths of cycles on which $g_i$ is and if $n=l_in_{g_i},$ one can take $p_{ii}$ to be the path obtaining by traversing a cycle of length $n_{g_i}$ $l_i$ times starting at $g_i$ so that $|p_{ii}|=n.$ Thus, relation (\ref{combinatorial_equation}) holds trivially for each $i$ since $m_i+n=m_i+n$ and so $a$ has a stationary-partition.  

\subsection{Characterization of comparability}\label{subsection_comparability}
Using Propositions \ref{connecting} and \ref{key_lemma}, we prove Theorem \ref{comparable} characterizing a comparable element. 

\begin{theorem}
The following conditions are equivalent for an element $a\in F^\Gamma_E.$
\begin{enumerate}[\upshape(1)]
\item The element $a$ is nonzero and comparable. 
\item There is a stationary element $b$ such that $a\to b.$  
\item There is an element $b$ with a stationary support and a stationary-partition such that condition (2) from Proposition \ref{connecting} holds for $a$ and $b.$  
\end{enumerate}
\label{comparable} 
\end{theorem}
\begin{proof}
The implication (1) $\Rightarrow$ (2) holds by Lemma \ref{comparable_connects_to_stationary}. Conversely, if (2) holds, then $b$ is nonzero and comparable. The relation $a\to b$ implies $a\sim b$ so $a$ is nonzero and comparable as well.

The equivalence (2) $\Leftrightarrow$ (3) follows directly from Propositions \ref{connecting} and \ref{key_lemma}.   
\end{proof}

In Theorem \ref{all_comparable}, we characterize when every element of $F_E^\Gamma$ is comparable. First, we show the following corollary of Proposition \ref{connecting} and Lemmas \ref{comparable_connects_to_stationary} and \ref{support_of_stationary} which we use in the proof of Theorem \ref{all_comparable}. 

\begin{corollary} Let $v$ be an infinite emitter. 
\begin{enumerate}[\upshape(1)] 
\item If $v$ connects to $q^v_Z$ by a path of positive length, then $v$ is on a cycle. 

\item If $q_Z^v$ connects to $q_W^v$ by a path of positive length, then $q_Z^v$ is on a cycle. 

\item If $q_W^v$ is on a cycle, then $v$ is on a cycle and $q^v_Z$ is on a cycle for every $\emptyset\neq Z\subseteq W.$ 

\item If $v$ is comparable, then $v$ is on a cycle. 

\item If $q^v_Z$ is comparable, then $q_Z^v$ is on a cycle.
\end{enumerate}
\label{infinite_emitter_corollary} 
\end{corollary}
\begin{proof}
To show (1), assume that $v\cto^p q_Z=q_Z^v$ for some path $p$ of positive length $n.$ Then $v\to x^n q_Z+a$ for some $a\in F_E^\Gamma.$ 
By the nature of axioms (A2) and (A3), there has to be a term $x^nv$ produced at some point. Hence, $v\to x^nv+b$ for some $b$ which implies that  
$v$ is on a cycle.  

To show (2), assume that $q_Z=q_Z^v\cto^p q_W=q_W^v$ for some path $p$ of positive length $n.$ Then $q_Z\to x^n q_W+a$ for some $a.$ 
By the nature of axioms (A2) and (A3), there has to be a term $x^nv$ produced at some point using a cycle $c$ based at $v$ such that the first edge of $c$ is not in $Z$. Hence, $q_Z\to x^nv+b\to x^nq_Z+c$ for some $b$ and $c$ by (A2) and so $q_Z$ is on a cycle.  

By Definition \ref{definition_connecting}, if $q_W^v$ is on a cycle, then there is a cycle based at $v$ such that the first edge $e$ of that cycle is not in $W.$ So, $v$ is on a cycle. If $Z\subseteq W,$ then $e\notin Z$ and so $q_Z$ is also on a cycle by Definition \ref{definition_connecting}. This shows (3).

To show (4), let $v$ be comparable. By Lemma \ref{comparable_connects_to_stationary}, there is a stationary element $a$ such that $v\to a$. If $a=v,$ then $v$ is stationary and it is necessarily on a cycle by  part (2) of Lemma \ref{support_of_stationary}. So, assume that $a\neq v.$ By Proposition \ref{connecting}, if $a=\sum_{j=1}^l x^{t_j}h_j$ is a normal representation of $a,$ there are paths $p_{j}, j=1,\ldots, l$ such that $v\cto^{p_j} h_j,$ $t_j=|p_j|,$ and at least one of $h_j$ is $q_Z^v$ for some $Z.$ Reordering the terms we can assume that $j=1.$ If $p_1$ has positive length, then $v$ is on a cycle by part (1). If $q_Z$ is on a cycle, then $v$ is on a cycle by part (3). So, let us consider the remaining case when $p_1$ is trivial and $q_Z$ is not on a cycle. In this case, $q_Z$ has to be on an exit from a core cycle by part (4) of Lemma \ref{support_of_stationary}. So, there is $j>1$ such that $h_j$ is on a cycle and $h_j\cto^p q_Z$ for some path $p.$ Hence, $v\cto^{p_j}h_j \cto^p q_Z.$ If $|p|>0,$ then $v$ connects to $q_Z$ by a path $p_jp$ of positive length and so $v$ is on a cycle by part (1). If $|p|=0,$ then either $h_j=v,$ in which case $v$ is on a cycle, or $h_j=q_{Z'}$ for some $Z'\subsetneq Z$ in which case $v$ is also on a cycle by part (3). 

To show (5), let $q_Z$ be comparable. By Lemma \ref{comparable_connects_to_stationary}, there is a stationary element $a$ such that $q_Z\to a.$ If $a=q_Z,$ then $q_Z$ is stationary and it is necessarily on a cycle by part (2) of Lemma \ref{support_of_stationary}. So, assume that $a\neq q_Z.$ By Proposition \ref{connecting}, if $a=\sum_{j=1}^l x^{t_j}h_j$ is a normal representation of $a,$ there are paths $p_{j}, j=1,\ldots, l$ such that $q_Z\cto^{p_j} h_j,$ $t_j=|p_j|,$ and at least one of $h_j$ is $q_W^v$ for some $W\supsetneq Z.$ Reordering the terms we can assume that $j=1.$ If $p_1$ has positive length, then $q_Z$ is on a cycle by part (2).  
If $q_W$ is on a cycle, then $q_Z$ is on a cycle by part (3). So, let us consider the remaining case when $p_1$ is trivial and $q_W$ is not on a cycle. By  part (4) of Lemma \ref{support_of_stationary}, there is $j>1$ such that $h_j$ is on a cycle and $h_j\cto^p q_W$ for some path $p.$ So, $q_Z\cto^{p_j}h_j \cto^p q_W.$ If $|p|>0,$ $q_Z$ connects to $q_W$ by a path $p_jp$ of positive length and so $q_Z$ is on a cycle by part (2). If $|p|=0,$ then either $h_j=v$ or $h_j=q_{Z'}$ for some $Z'\subsetneq W.$ In the first case, $q_Z\cto^{p_j}v$ so there is a cycle based at $v$ such that its first edge $e$ is not in $Z$ and so $q_Z$ is on that cycle by Definition \ref{definition_connecting}. In the second case, $q_Z\cto^{p_j}q_{Z'}.$ If $|p_j|>0,$ then $q_Z$ is on a cycle by part (2). If $|p_j|=0,$ then $Z\subseteq Z'.$ Since $q_{Z'}$ is on a cycle, there is a cycle based at $v$ such that the first edge $e$ of it is in $\so^{-1}(v)-Z'.$ Hence, $e\notin Z$ and so $q_Z$ is on a cycle by Definition \ref{definition_connecting}. 
\end{proof}

\begin{theorem}
The following conditions are equivalent.  
\begin{enumerate}[\upshape(1)]
\item Every element $a\in F_E^\Gamma$ is comparable.
\item Every generator of $F_E^\Gamma$ is comparable. 
\item For every generator $g$ of $F_E^\Gamma,$ $g\to a$ for some stationary element $a.$  
\item The following hold for every generator $g$ of $F_E^\Gamma.$ 
\begin{enumerate}[\upshape(a)]
\item The generator $g$ is not a sink and it connects to a cycle. 
\item If $g$ is an infinite emitter or an improper vertex, then $g$ is on a cycle.
\item If $g$ is regular, there is stationary $a\in F_E^\Gamma$ with the exit part zero such that $g\to a.$   
\end{enumerate}

\item The following hold for every vertex $v$ of $E.$ 
\begin{enumerate}[\upshape(a)]
\item The vertex $v$ is not a sink and it connects to a cycle. 
\item If $v$ is an infinite emitter, then it is on a cycle. 
\item If $v$ is regular, there are finitely many proper or improper vertices $h_1,\ldots, h_l$ on cycles and paths $p_j, j=1,\ldots, l$ such that $v\cto^{p_j}h_j$ and such that for every prefix $p$ of $p_j$ the conditions (i), (ii) and (iii) of Proposition \ref{connecting} hold with $t_j=|p_j|.$ 
\end{enumerate}
\end{enumerate}
\label{all_comparable}
\end{theorem}
\begin{proof}
The implication (1) $\Rightarrow$ (2) is direct and the implication (2) $\Rightarrow$ (1) holds since a finite sum of comparable elements is comparable. The equivalence (2) $\Leftrightarrow$ (3) follows directly from Theorem \ref{comparable}. To complete the proof, we show (3) $\Rightarrow$ (4) $\Rightarrow$ (5) $\Rightarrow$ (3). 

Assume that (3) holds and let $g$ be any generator. Let $a$ be stationary such that $g\to a.$ Since $g$ connects to all generators in the support of $a_c\neq 0,$ $g$ connects to a generator on a cycle so $g$ is not a sink and (a) holds. Part (b) holds by parts (4) and (5) of Corollary \ref{infinite_emitter_corollary}. To show part (c), let $g=v\in E^0$ be regular. We claim that there is an element $b\in F_E^\Gamma-\{0\}$ with support containing only vertices on cycles such that $v\to b$.
As $v\to a$ for a stationary element $a,$ $v$ cannot be on an infinite path containing no vertices on cycles by Proposition \ref{connecting} and Lemma \ref{support_of_stationary}. In addition, a path from $v$ to a cycle such that only the range of that path is in a cycle can contain an  infinite emitter only as its range. This shows that there are only finitely many paths $p$ originating at $v$ and terminating in a vertex of a cycle such that no other vertex of $p$ but $\ra(p)$ is in a cycle. Let $n_v$ be the maximal element of the set of lengths of such paths. We show the claim by induction on $n_v.$ If $n_v=0,$ $v$ is on a cycle and one can take $b=v.$ Assuming the induction hypothesis, consider a regular vertex $v$ with $n_v>0.$ For every $e\in \so^{-1}(v),$ either $\ra(e)$ is on a cycle, in which case we let $b_e=\ra(e)$ or $\ra(e)$ is not on a cycle in which case $\ra(e)$ is necessarily regular by parts (a) and (b) and with $n_{\ra(e)}<n_v.$ By the induction hypothesis, there is an element $b_e$ with the support consisting of vertices on cycles such that $\ra(e)\to b_e$. The element $b=\sum_{e\in  \so^{-1}(v)}xb_e$ has vertices in the support on cycles and 
\[v\to_1 \sum_{e\in  \so^{-1}(v)}x\ra(e)\to \sum_{e\in  \so^{-1}(v)}xb_e=b.\]
Since $\supp(b)$ consists of generators on cycles, $b$ is stationary by Lemma \ref{stationary_support_cyclic} and its exit part is zero. 

Assume that (4) holds and let $v$ be any vertex of $E$. Parts (5a) and (5b) directly hold by (4a) and (4b). If $v$ is regular, let $a$ be stationary with exit part zero such that $v\to a$ which exists by (4c). If $a=\sum_{j=1}^l x^{t_j}h_j,$ then $h_j$ are on cycles since the exit part of $a$ is zero. Part (5c) then follows from the relation $v\to a$ by Proposition \ref{connecting}.  

Assume that (5) holds and let $g$ be any generator. By (5a), $g$ is not a sink. If $g$ is an infinite emitter, then $g$ is on a cycle by (5b) and so it is stationary. If $g$ is an improper vertex and $g=q^v_Z,$ then 
$v$ is on a cycle by (5b) so $g$ is on a cycle by Definition \ref{definition_connecting} and, again $g$ is stationary. In both of these cases, (3) holds since $g\to g.$ If $g$ is a regular vertex, (5c) and  Proposition \ref{connecting} imply that $g\to a$ for $a=\sum_{j=1}^l x^{|p_j|}h_j.$ Since the elements $h_j$ are on cycles, $a$ is stationary by Lemma \ref{stationary_support_cyclic}. Hence, (3) holds. 
\end{proof}

Part (5) with any of the conditions (a), (b), or (c) deleted is not equivalent with the other conditions of Theorem \ref{all_comparable} as the next set of examples shows.
\begin{example}
\begin{enumerate}
\item If $E$ is the Toeplitz graph (see part (2) of Example \ref{example_stationary}), then (b) and (c) hold. There is a sink so (a) fails and the sink is not comparable. 

\item If $E$ is the graph below, then (a) and (c) hold. The infinite emitter $v$ is not on a cycle, so (b) fails and $v$ is not comparable by Corollary \ref{infinite_emitter_corollary}(4).  

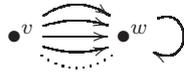
$$\xymatrix{{\bullet}^{v} \ar@{.} @/_1pc/ [r] _{\mbox{ } } \ar@/_/ [r] \ar [r] \ar@/^/ [r] \ar@/^1pc/ [r] & {\bullet}^{w}\ar@(ru,rd)}$$

\item Let $E$ be the graph below. 

$$
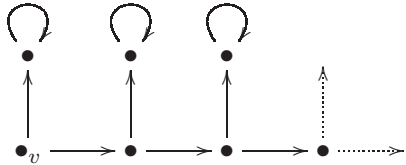
\xymatrix{   
\bullet  \ar@(ul,ur)   & \bullet \ar@(ul,ur)    & \bullet  \ar@(ul,ur) &   & \\   
\bullet_v \ar[r]\ar[u] & \bullet \ar[r]  \ar[u] & \bullet \ar[r]\ar[u] & \bullet \ar@{.>}[r] \ar@{.>}[u] &}$$
If $a$ is any element whose support consists only of vertices on cycles, then $v\to a$ fails since  
there is a path originating at $v$ which does not connect to $\supp(a)$ (analogous argument is used in part (3) of Example \ref{example_connecting}). 
The conditions (a) and (b) hold for $E$, but (c) fails and $v$ is not comparable. 
\end{enumerate}
\label{example_not_comparable}
\end{example}

\section{Characterizations of periodic, aperiodic and incomparable elements}\label{section_other_three}

Next, we show characterizations of periodic, aperiodic and incomparable elements as well as other properties discussed in the introduction. We start by Theorem \ref{periodic} which characterizes a nonzero periodic element of $F_E^\Gamma$. Theorem \ref{periodic} has already been used in \cite[Theorem 3.1]{Crossed_product} to characterize Leavitt path algebras which are crossed products in terms of the properties of the underlying graphs.

\begin{theorem} The following conditions are equivalent for an element $a\in F^\Gamma_E-\{0\}.$
\begin{enumerate}[\upshape(1)]
\item The element $a$ is periodic. 
\item There is an element $b$ whose support consists of vertices on cycles without exits such that $a\sto b.$
\item Any path originating at a generator in the support of $a$ is a prefix of a path $p$ ending in one of finitely many cycles with no exits and such that all vertices of $p$ are regular. Every infinite path originating at a vertex in the support of $a$ ends in a cycle with no exits.
\end{enumerate}
\label{periodic}
\end{theorem}
\begin{proof}
If (1) holds, then $a$ is comparable so $a\to b$ for some stationary element $b$ by Lemma \ref{comparable_connects_to_stationary}. The relation $a\to b$ implies $a\sim b$ so $b$ is periodic as well. Hence, the supports of both $a$ and $b$ consists of regular vertices only by Lemma \ref{lemma_three_claims}. Thus, $a\sto b.$ By Corollary \ref{stationary_periodic}, the support of $b$ consists of regular vertices on cycles without exits which shows (2).

If (2) holds, the element $b$ as in (2) is stationary and periodic by Lemma \ref{stationary_support_cyclic}. Since the core cycles of $b$ do not have exits, each generator in $\supp(b)$ is proper, emits exactly one edge and, hence, it is regular. As $a\sto b,$ any element of $\supp(a)$ is proper and regular also. Let $a=\sum_{i=1}^k x^{m_i}v_i,$ $b=\sum_{j=1}^l x^{t_j}w_j,$ and $I_i$ and $p_{ij}$ be as in Proposition \ref{connecting} for $a\to b.$ Since only (A1) is used, each vertex of any path $p_{ij}$ is regular. 

If $p$ is a path with $\so(p)=v_i,$ we use induction on $|p|$ to show that there is a path $q$ such that $p$ is a prefix of $q,$ $q$ ends in one of the core cycles and all vertices of $q$ are on some $p_{ij}$ for $j\in I_i$ (thus regular) or on cycles without exits (thus also regular). 
If $p=v_i,$ $q$ can be taken to be $p_{ij}$ for any $j\in I_i.$ Assuming that the claim holds for $p,$ let us consider $pe$ for some edge $e.$ By the induction hypothesis, all vertices of $p$ are regular, on $p_{ij}$ for some $j\in I_i$ or on a core cycle. If $\ra(e)$ is on a core cycle, then it emits exactly one edge so it is regular and we can take $q$ to be $pe.$ So, let us consider the case that $\ra(e)$ is not on a core cycle in which case  $\ra(p)$ is not on a core cycle also and so $\ra(p)$ is on $p_{ij}$ for some $j\in I_i.$ 
Since $\ra(p)$ is not on a cycle, there is a proper prefix $r$ of $p_{ij}$ which ends in $\ra(p).$ Thus $P_r\neq\emptyset$ and so $P_r=\so^{-1}(\ra(p))$ by part (2)(i) of Proposition \ref{connecting}. In particular, $e\in P_r.$ Hence, there is $j'\in I_i$ such that $e$ is in $p_{ij'}.$ Let $q$ be $pe$ up to $\ra(e)$ and the suffix of $p_{ij'}$ after $pe.$ Thus, $pe$ is a prefix of $q,$ $q$ ends in a core cycle and each vertex of $q$ is on $p_{ij}$ for some $j\in I_i.$   

It remains to show the condition on the infinite path. Let $e_1e_2\ldots$ be an infinite path originating at $v_i.$ For any $n,$ each vertex of the path $e_1e_2\ldots e_n$ is on $p_{ij}$ for some $j\in I_i$ or in a core cycle. Let $n$ be strictly larger than the length of $p_{ij}$ for all $j\in I_i.$ Then $\ra(e_n)$ must be in a core cycle and so $e_ne_{n+1}\ldots$ is on that same cycle since the cycle has no exits. This shows that (3) holds.  

If condition (3) holds, then the support of $a$ consists of regular vertices such that every path they emit connects to finitely many cycles without exits by paths which contain regular vertices only. Let $\supp(a)=\{v_1,\ldots, v_k\}$ and let $n_i$ be the number of paths $p$ from $v_i$ to the finitely many cycles from condition (3) such that no vertex of any of the paths from (3) is on the cycle except the range of $p.$ Index the paths  originating at $v_i$ as $p_{i1},\ldots, p_{in_i}$ for some positive $n_i$ and let $w_{ij}=\ra(p_{ij}).$ Let $J$ be the set of  $(i, j)$ with $i=1,\ldots, k$ and $j=1,\ldots, n_i$ and let $I_i$ be the set of those $(i',j)\in J$ such that $i'=i.$ By construction, $\{I_1,\ldots, I_k\}$ is a partition of $J$ and, by considering a bijection between $J$ and the set $\{1,\ldots, l\}$ for $l=|J|,$ this partition corresponds to a partition of  $\{1,\ldots, l\}.$ 

If $p$ is a prefix of $p_{ij},$ let us use the notation $P_p$ in the same sense as in Proposition \ref{connecting} and Definition \ref{definition_connecting_support}. If $p$ is a proper prefix of $p_{ij},$ then $\ra(p)$ is regular and $P_p$ is nonempty as it contains the first edge of $p_{ij}$ not on $p.$ If $e\in \so^{-1}(\ra(p)),$ then $pe$ is a prefix of some $p_{ij'}$ by condition (3) and so $e\in P_p.$ Hence, $P_p=\so^{-1}(\ra(p)).$ If $p=p_{ij},$ then $P_p$ is empty by construction. Thus condition (i) of Definition \ref{definition_connecting_support} holds and conditions (ii) and (iii) are trivially satisfied. So, for $W=\{w_{ij}\mid  (i,j)\in J\},$ $\supp(a)\to W.$ Moreover, $\supp(a)\sto W$ since $\supp(a)$ and $W$ contain regular vertices only. Hence, there is $c\neq 0$ such that $a\sto c$ and $\supp(c)\subseteq W$ by Corollary \ref{strong_connecting}. The set $W$ is stationary and, by part (1) of Lemma \ref{stationary_support_cyclic}, every element with support contained in $W$ is stationary and, part (2) of Lemma \ref{stationary_support_cyclic}, periodic. Thus, $c$ is periodic. Since $a\sim c,$ $a$ is also periodic. Hence, (1) holds.   
\end{proof}

We note that the sources of graphs in parts (1) and (3) of Example \ref{example_not_comparable} are such that condition (2) fails, so that these vertices are not periodic by Theorem \ref{periodic} (and incomparable by Theorem \ref{comparable}). 

In Theorem \ref{all_periodic}, we characterize when every element of $F_E^\Gamma$ is periodic in terms of the properties of $E$, in terms of the form of the Leavitt path algebra, as well as in terms of the form of the Grothendieck $\Gamma$-group. 

\begin{theorem}
The following conditions are equivalent.  
\begin{enumerate}[\upshape(1)]
\item Every element $a\in F_E^\Gamma$ is periodic.
\item Every vertex is periodic. 
\item For every vertex $v,$ $\{v\}\sto V$ for some stationary set $V$ which contains core vertices only and every core cycle has no exits.
\item Each path is a prefix of a path $p$ ending in a cycle with no exits and such that the vertices on $p$ are regular. Every infinite path ends in a cycle with no exits. 
\item $E$ is a row-finite, no-exit graph without sinks such that every infinite path ends in a cycle. 
\item For any field $K$, the Leavitt path algebra $L_K(E)$ is graded isomorphic to an algebra of the form 
\[\bigoplus_{i\in I}\M_{\mu_i}(K[x^{n_i}, x^{-n_i}])(\ol\gamma_i)\]
where $I$ is a set, $\mu_i$ are cardinals, $n_i$ positive integers, and $\ol\gamma_i$ maps $\mu_i\to \Zset$ for $i\in I.$  
\item The graph $\Gamma$-monoid is isomorphic to  
\[\bigoplus_{i\in I} \Zset^+[x]/\langle x^{n_i}=1\rangle\]
where $I$ is a set and $n_i$ are positive integers for $i\in I$.
\item The  Grothendieck $\Gamma$-group $G_E^\Gamma$ is isomorphic to  
\[\bigoplus_{i\in I} \Zset[x]/\langle x^{n_i}=1\rangle\]
where $I$ is a set and $n_i$ are positive integers for $i\in I$.
\end{enumerate}
\label{all_periodic}
\end{theorem}
\begin{proof}
The implication (1) $\Rightarrow$ (2) is direct. If (2) holds, then every vertex of $E$ is regular by Lemma \ref{lemma_three_claims}. If a vertex $v$ is periodic,  $v\to a$ for some stationary element $a$ by Lemma \ref{comparable_connects_to_stationary}. Since $v$ is periodic, $a$ is periodic also and  the support of $a$ consists of regular vertices on cycles without exits by Corollary \ref{stationary_periodic}. Thus, $v\to a$ implies that $v\sto a.$ If $V=\supp(a),$ condition (3) follows by Corollary \ref{strong_connecting}. 

If (3) holds, then all vertices of $E$ are regular. If $p$ is any finite or infinite path, $\{\so(p)\}\sto V$ for some $V$ as in condition (3). By Corollary \ref{strong_connecting}, there is $a\in F_E^\Gamma-\{0\}$ such that $\so(p)\sto a$ and $\supp(a)\subseteq V.$ Since $V$ consists of vertices on cycles without exits, $a$ is stationary and periodic and so $\so(p)$ is periodic also. Then (4) holds by Theorem \ref{periodic}.

If (4) holds, then all vertices of $E$ are regular so $E$ is a row-finite graph. Every vertex connects to cycles so there are no sinks. Every infinite path ends in a cycle and no cycle has an exit. So, (5) holds.

Conditions (5) and (6) are equivalent by \cite[Corollary 3.6]{Lia_no-exit}.

The implications (6) $\Rightarrow$ (7) and (7) $\Rightarrow$ (8) are rather direct. Condition (8) directly implies that every element of $G_E^\Gamma$ has a finite orbit. Hence, every element of $F_E^\Gamma$ is periodic and (1) holds. 
\end{proof}

Using Theorem \ref{periodic}, we characterize when no nonzero element of $F_E^\Gamma$ is periodic. 

\begin{corollary}
The following conditions are equivalent.  
\begin{enumerate}[\upshape(1)]
\item No nonzero element of $F_E^\Gamma$ is periodic.
\item The graph $E$ satisfies Condition (L). 
\end{enumerate}
\label{no_periodic} 
\end{corollary}
\begin{proof}
If $E$ has a cycle with no exits, any vertex on this cycle is periodic. Conversely, if Condition (L) holds, the core cycles of any stationary element have exits. By Theorem \ref{periodic}, no nonzero element is periodic. 
\end{proof}

We characterize aperiodic elements next. 

\begin{theorem}
The following conditions are equivalent for an element $a\in F^\Gamma_E.$
\begin{enumerate}[\upshape(1)]
\item The element $a$ is aperiodic.
\item The element $a$ is comparable and not periodic. 
\item There is a stationary element $b$ such that $a\to b$ and at least one of the core cycles of $b$ has an exit.  
\end{enumerate}
\label{aperiodic}
\end{theorem}
\begin{proof}
It is direct that (1) $\Leftrightarrow$ (2). The equivalence (2) $\Leftrightarrow$ (3) holds by Theorems \ref{comparable} and \ref{periodic}. 
\end{proof}

We also characterize when every element of $F_E^\Gamma$ is aperiodic.

\begin{theorem}
The following conditions are equivalent.  
\begin{enumerate}[\upshape(1)]
\item Every nonzero element $a\in F_E^\Gamma$ is aperiodic. 
\item Every generator of $F_E^\Gamma$ is aperiodic.  
\item Every generator of $F_E^\Gamma$ is comparable and every cycle has an exit. 
\item For every generator $g$ of $F_E^\Gamma,$ $g\to a$ for some stationary element $a$ such that all core cycles have exits. 
\end{enumerate}
\label{all_aperiodic}
\end{theorem}
\begin{proof}
The implication (1) $\Rightarrow$ (2) is direct. The converse holds since a sum of aperiodic elements is comparable and, if at least one of them is aperiodic, aperiodic. 

If (2) holds and $g$ is a generator on an arbitrary cycle (which exists by Corollary \ref{exists_comparable}), then $g$ is aperiodic if and only if the cycle has an exit by Lemma \ref{stationary_support_cyclic}. Hence, (3) holds. 

If (3) holds and $g$ is an arbitrary generator, then $g\to a$ for a stationary element $a.$ By assumption (3) all core cycle of $a$ have exits so (4) holds. 

Finally, let us assume that (4) holds and show (2). If $g$ is an arbitrary generator and $a$ a stationary element such that $g\to a$ and all core cycles have exit, then $a\to x^na+b$ for some nonzero $b$. Hence, $a$ is aperiodic and, since $g\to a,$ $g$ is aperiodic also.   
\end{proof}

We also characterize when no element of $F_E^\Gamma$ is aperiodic.   

\begin{corollary}
The following conditions are equivalent.  
\begin{enumerate}[\upshape(1)]
\item No element of $F_E^\Gamma$ is aperiodic.
\item The graph $E$ is no-exit (i.e. satisfies Condition (NE)). 
\end{enumerate}
\label{no_aperiodic} 
\end{corollary}
\begin{proof}
If $E$ is not a no-exit graph, there is a cycle with an exit and any vertex on that cycle is an aperiodic element of $F_E^\Gamma.$  
Conversely, if $a$ is an aperiodic element of $F_E^\Gamma,$ then $a\to b$ for some stationary element $b$ such that at least one of the core cycles of $b$ must have an exit by Theorem \ref{aperiodic}. Hence, $E$ is not no-exit. 
\end{proof}

Since every element which is not comparable is incomparable, Theorem \ref{comparable} implies a characterization of an incomparable element in $F_E^\Gamma$ also. The following characterization of graphs such that all elements of $F_E^\Gamma$ are incomparable follows directly from Corollary \ref{exists_comparable}. 
 
\begin{corollary}
The following conditions are equivalent.  
\begin{enumerate}[\upshape(1)]
\item Every nonzero element $a\in F_E^\Gamma$ is incomparable. 
\item Every generator of $F_E^\Gamma$ is incomparable.
\item The graph $E$ is acyclic. 
\end{enumerate}
\label{all_incomparable}
\end{corollary}

\subsection{Strengthening results of \texorpdfstring{\cite{Talented_monoid}}{TEXT}}\label{subsection_talented}

By Proposition \ref{generalization_of_lemma_4_1}, a result of \cite{Talented_monoid} holds without the assumption that the graph under consideration is row-finite. In this section, we show that the same assumption can be deleted from some of the main results of \cite{Talented_monoid}. The second part of Corollary \ref{talented4} shows that our results provide some further progress towards a positive answer to the Graded Classification Conjecture.

First, we show that Theorems \ref{periodic} and \ref{aperiodic} and Corollary \ref{all_incomparable} imply \cite[Proposition 4.2]{Talented_monoid} without assuming that the graph is row-finite. We formulate this in the following corollary. 

\begin{corollary} 
\begin{enumerate}[\upshape(1)]
\item The graph $E$ has a cycle with no exit if and only if some nonzero element of $F^\Gamma_E$ is periodic. 
 
\item The graph $E$ has a cycle with an exit if and only if some element of $F^\Gamma_E$ is aperiodic. 

\item The graph $E$ is acyclic if and only if every nonzero element of $F^{\Gamma}_E$ is incomparable. 
\end{enumerate}
\label{talented_corollary1} 
\end{corollary}
\begin{proof}
One direction of parts (1) and (2) follows by Theorems \ref{periodic} and \ref{aperiodic}. The other follows by Lemma \ref{stationary_support_cyclic} which implies that a vertex on a cycle is periodic if the cycle has no exits and it is aperiodic if the cycle has an exit. Part (3) directly follows from Corollary \ref{all_incomparable}.    
\end{proof}

By \cite[Theorem 5.7]{Tomforde}, a $\Gamma$-order-ideal of $M_E^\Gamma$ uniquely determines certain subset of vertices. We briefly review this construction. A subset $H$ of $E^0$ is said to be {\em hereditary} if for any $v\in H$ and a path $p$ with $\so(p)=v,$ $\ra(p)$ is in $H$ and it is  {\em saturated} if $\ra(\so^{-1}(v))\subseteq H$ for a regular vertex $v$ implies that $v\in H.$   

For a hereditary and saturated set $H$, let 
\[G(H)=\{v\in E^0-H\mid v\mbox{ is not regular and }\so^{-1}(v)\cap \ra^{-1}(E^0-H)\mbox{ is nonempty and finite} \}.\]
For $G\subseteq G(H),$ the pair $(H, G)$ is said to be an {\em admissible pair}. 
The set of all such pairs is a lattice by 
\[(H_1, G_1)\leq (H_2, G_2)\; \mbox{ iff }\; H_1\subseteq H_2,\;\; G_1\subseteq G_2\cup H_2\]
(see \cite{Tomforde} or \cite{Ara_Goodearl}). By \cite[Theorem 5.7]{Tomforde}, this lattice is isomorphic to the lattice of graded ideals of $L_K(E)$ and by \cite[Theorem 6.9]{Ara_Goodearl}, this lattice is isomorphic to the set of order-ideals of $M_E.$ If $(H,G)\mapsto I(H,G)$ denotes this isomorphism, then $M_E/I(H,G)\cong M_{E/(H,G)}$ 
and both \cite{Tomforde} and \cite{Ara_Goodearl} contain details. By \cite[Lemma 5.10]{Ara_et_al_Steinberg}, the lattices of order-ideals of $M_E$ and of $\Gamma$-order-ideals of $M_E^\Gamma$ are isomorphic. 
Moreover, if the assumption that $E$ is row-finite is deleted and hereditary and saturated set replaced by an admissible pair, the proof of \cite[Lemma 2.2]{Talented_monoid} establishes that 
\[M^\Gamma_E/I(H,G)\cong M^\Gamma_{E/(H,G)}\] for an admissible pair $(H,G).$

Next, we show that the assumption that $E$ is row-finite can be removed from \cite[Corollary 4.3]{Talented_monoid}.

\begin{corollary}  
\begin{enumerate}[\upshape(1)]
\item The following conditions are equivalent. 
\begin{enumerate}[\upshape(i)]
\item The graph $E$ satisfies Condition (L).
\item No nonzero element of $F_E^\Gamma$ is periodic.
\item $\Gamma$ acts freely on $M_E^\Gamma.$ 
\end{enumerate}
\smallskip
\item The following conditions are equivalent.  
\begin{enumerate}[\upshape(i)]
\item The graph $E$ satisfies Condition (K).
\item No nonzero element of $M_E^\Gamma/I$ is periodic for any $\Gamma$-order-ideal $I$ of $M_E^\Gamma.$
\item The group $\Gamma$ acts freely on $M_E^\Gamma/I$ for any $\Gamma$-order-ideal $I$ of $M_E^\Gamma.$ 
\end{enumerate}
\end{enumerate}
\label{talented_corollary2} 
\end{corollary}
\begin{proof}
Part (1) directly follows from Corollary \ref{no_periodic}.  

By \cite[Proposition 6.12]{Tomforde}, $E$ satisfies Condition (K) if and only if $E/(H, G)$ satisfies Condition (L) for any admissible pair $(H,G).$ Since every such pair uniquely determines a $\Gamma$-order-ideal of $M_E^\Gamma,$ part (1) and Corollary \ref{no_periodic} imply the equivalences of conditions in part (2). 
\end{proof}

\cite[Corollary 5.1]{Talented_monoid} focuses on the monoid properties of $M_E^\Gamma$ which are equivalent with various forms of simplicity of $L_K(E).$ We show these properties without requiring that $E$ is row-finite.  

\begin{corollary}
Let $K$ be any field. 
\begin{enumerate}[\upshape(1)]
\item The following conditions are equivalent. 
\begin{enumerate}[\upshape(i)]
\item The algebra $L_K(E)$ is graded simple. 

\item The $\Gamma$-monoid $M_E^\Gamma$ is simple.

\item The $\Gamma$-group $G_E^\Gamma$ is simple as an ordered $\Gamma$-group.  
\end{enumerate}
\smallskip 
\item  The following conditions are equivalent. 
\begin{enumerate}[\upshape(i)]
\item The algebra $L_K(E)$ is simple. 

\item The $\Gamma$-monoid $M_E^\Gamma$ is simple and no nonzero element of $M_E^\Gamma$ is periodic. 

\item The $\Gamma$-monoid $M_E^\Gamma$ is simple and every nonzero comparable element of $M_E^\Gamma$ is aperiodic. 
\end{enumerate}
\smallskip
\item  The following conditions are equivalent. 
\begin{enumerate}[\upshape(i)]
\item The algebra $L_K(E)$ is purely infinite simple. 
\item The $\Gamma$-monoid $M_E^\Gamma$ is simple, no nonzero element of $M_E^\Gamma$ is periodic and some element of $M_E^\Gamma$ is aperiodic. 
\end{enumerate}
\end{enumerate}
\label{talented_corollary3}
\end{corollary}
\begin{proof}
Part (1) directly follows from the fact that the lattices of graded ideals of $L_K(E),$ $\Gamma$-order-ideals of $M_E^\Gamma$ and $\Gamma$-order-ideals of $G_E^\Gamma$ are isomorphic. 

By \cite[Theorem 2.9.1]{LPA_book}, $L_K(E)$ is simple if and only if it is graded simple and $E$ satisfies Condition (L). By part (1) and Corollary \ref{no_periodic}, this is equivalent with $M_E^\Gamma$ being simple and without a nonzero periodic element. This last condition is equivalent with the requirement that every nonzero comparable element is aperiodic.   

By \cite[Theorem 3.1.10]{LPA_book},  $L_K(E)$ is purely infinite simple if and only if it is simple and $E$ has a cycle with an exit. By Corollary \ref{talented_corollary1}, $E$ has a cycle with an exit if and only if $M_E^\Gamma$ has an aperiodic element.
\end{proof}

Lastly, we show Corollary \ref{talented4}. Parts (1) and (3) show that the first part of \cite[Theorem 5.7]{Talented_monoid} holds without the condition that $E$ is row-finite. Parts (4) to (8) are further corollaries of our results.  

\begin{corollary}
Let $E$ and $F$ be arbitrary graphs. If there is a $\Gamma$-monoid isomorphism $M_E^\Gamma\to M_F^\Gamma,$ then the following hold.
\begin{enumerate}[\upshape(1)]
\item The graph $E$ satisfies Condition (L) if and only if $F$ satisfies Condition (L).
 
\item The graph $E$ satisfies Condition (K) if and only if $F$ satisfies Condition (K).

\item The lattices of graded ideals of $L_K(E)$ and $L_K(F)$ are isomorphic. 

\item $E$ is acyclic if and only if $F$ is acyclic. 

\item There is a cycle without an exit in $E$ if and only if there is a cycle without an exit in $F.$

\item There is a cycle with an exit in $E$ if and only if there is a cycle with an exit in $F.$ 

\item None of the cycles of $E$ have exits if and only if none of the cycles of $F$ have an exit.  

\item $E$ satisfies the condition below if and only if $F$ satisfies the condition below. 
\begin{itemize}
\item[] The graph is row-finite, no-exit, has no sinks and it is such that every infinite path ends in a cycle. 
\end{itemize}
\end{enumerate}
\label{talented4} 
\end{corollary}
\begin{proof}
Parts (1) and (2) directly follow from Corollary \ref{talented_corollary2}. To show part (3), note that a $\Gamma$-monoid isomorphism $M_E^\Gamma\to M_F^\Gamma$ induces a lattice isomorphism on the lattices of $\Gamma$-order-ideals. Since these lattices are isomorphic to lattices of graded ideals of $L_K(E)$ and $L_K(F),$ part (3) holds. 

Part (4) holds since $E$ has a cycle if and only if there is a nonzero comparable element in $M_E^\Gamma$ by Corollary \ref{exists_comparable}.
Part (5) holds since $E$ has a cycle with no exit if and only if there is a nonzero periodic element in $M_E^\Gamma$ by Corollary \ref{talented_corollary1}(1). 
Part (6) holds since $E$ has a cycle with an exit if and only if there is an aperiodic element in  $M_E^\Gamma$ by Corollary \ref{talented_corollary1}(2).

Part (7) holds by Corollary \ref{no_aperiodic} and part (8) by Theorem \ref{all_periodic}.
\end{proof}

Corollary \ref{talented4} asserts that many relevant properties of two graphs match if the graphs have isomorphic graph $\Gamma$-monoids. Together with our previous results, Corollary \ref{talented4} indicates that the Graded Classification Conjecture may have a positive answer since the properties of the graph are well reflected by the structure of its graph $\Gamma$-monoid. 

The Graded Classification Conjecture was shown for finite polycephaly graphs in \cite{Roozbeh_Annalen} and for a certain class of countable, row-finite, no-exit graphs in \cite{Roozbeh_Lia_Ultramatricial}. In \cite{Eilers_et_al}, it was shown for countable graphs such that for any two vertices the set of edges from one to the other is either empty or infinite. We also note that a weaker version of the conjecture was shown for finite graphs with neither sources nor sinks in \cite{Ara_Pardo}.

\end{document}